\documentclass[pdflatex,sn-mathphys]{sn-jnl}

\jyear{2021}%

\theoremstyle{thmstyleone}%
\newtheorem{theorem}{Theorem}
\newtheorem{proposition}[theorem]{Proposition}%
\newtheorem{lemma}[theorem]{Lemma}
\usepackage{setspace,amsmath}

\theoremstyle{thmstyletwo}%
\newtheorem{example}{Example}%
\newtheorem{remark}{Remark}%

\theoremstyle{thmstylethree}%
\newtheorem{definition}{Definition}%

\begin{document}

\title[Bilevel Programming Problems: A view through Set-valued Optimization]{Bilevel Programming Problems: A view through Set-valued Optimization }


\author*[1]{\fnm{Kuntal} \sur{Som}}\email{kuntal24.som@gmail.com}


\author[2]{\fnm{Thirumulanathan} \sur{D.}}\email{nathan@iitk.ac.in}
\equalcont{These authors contributed equally to this work.}

\author[2]{\fnm{Joydeep} \sur{Dutta}}\email{jdutta@iitk.ac.in}
\equalcont{These authors contributed equally to this work.}

\affil[1]{\orgdiv{Department of Mathematics}, \orgname{IIT Jodhpur}, \orgaddress{\street{NH62, Surpura Bypass Road, Karwar }, \city{Jodhpur}, \postcode{342030}, \state{Rajasthan}, \country{India}}}

\affil[2]{\orgdiv{Department of Economic Sciences}, \orgname{IIT Kanpur}, \orgaddress{\street{Kalyanpur}, \city{Kanpur}, \postcode{208016}, \state{Uttar Pradesh}, \country{India}}}



\abstract{Bilevel programming is one of the very active areas of research with many real-life applications in economics and engineering. Bilevel problems are hierarchical problems consisting of lower-level and upper-level problems, respectively. The leader or the decision-maker for the upper-level problem decides first, and then the follower or the lower-level decision-maker chooses his/her strategy. In the case of multiple lower-level solutions, the bilevel problems are not well defined, and there are many ways to handle such a situation. One standard way is to put restrictions on the lower level problems (like strict convexity) so that nonuniqueness does not arise. However, those restrictions are not viable in many situations. Therefore, there are two standard formulations, called pessimistic formulations and optimistic formulations of the upper-level problem (see \cite{Dempe2002foundation}). A set-valued formulation has been proposed in the paper \cite{Dutta2017optimalityofbilevelthroughvariational} and has been studied in the literature (see \cite{Zemkohoillposed,pileckathesis}). However, the study is limited to the continuous set-up with the assumption of value attainment, and the general case has not been considered. In this paper, we focus on the general case and study the connection among various notions of solution. Our main findings suggest that the set-valued formulation may not hold any bigger advantage than the existing optimistic and pessimistic formulation. }




\maketitle

\section{Introduction}
\label{intro}
Bilevel programming is a very active area of research. It originated in the work of Stackelberg \cite{stackelberg1952theory} in the 1930s and is still very popular with lots of applications. One can have a look at \cite{Bard,Bard1988,dempe1992,Dempe2002foundation,dempe2006dutta,Dempe2012bilevelandcomplimentary}. Very briefly, the bilevel programming can be described as follows. There are two agents; one is called the leader, while the other is called the follower. The leader provides some input to the follower, and the follower then minimizes or maximizes his/her objective by considering the leader's input as a parameter. Therefore, the leader has to choose an input so as to minimize or maximize his/her objective by taking into consideration the follower's optimal choices. So there is a hierarchy: there are two levels of problems, and these are named lower-level and upper-level problems, respectively. The key challenge in bilevel programming is that it is an intrinsically nonsmooth and nonconvex optimization problem. In the case of multiple lower-level solutions, the bilevel problems are not well defined, and two standard formulations in terms of optimistic and pessimistic scenarios have been proposed to handle ambiguity in such cases. A set-valued formulation has been proposed in \cite{Dempedutta2006bilevelwithconvexlower} (also see \cite{Dutta2017optimalityofbilevelthroughvariational}), and has been further studied in \cite{pileckathesis,Zemkohoillposed} in the continuous setup. However, the general setup has not been studied yet.  

On the other hand, set-valued optimization is a very active area of research. Originating in the late 1970s, it has been growing to date with many diverse applications in different fields of engineering and sciences \cite{borwein1977multivalued,corley1988optimality,DK4,Akhtar,Gofert}. Initially, a solution concept called the `vector approach' to a set-valued optimization problem was prominent \cite{borwein1977multivalued,corley1987existence,corley1988optimality,Luc89}. However, some shortcomings of this approach were pointed out in the late 1990s by Kuroiwa et al. \cite{DK}. They initiated and popularized the `set-relation approach' of solutions to a set-valued optimization problem \cite{KurTanTru97, DK4}. Since then, research in set-valued optimization in the set-relation approach has been booming. Set-valued optimization has been successfully used as an explanatory tool for games with multi-objective payoffs in \cite{H}. In this paper, we focus on the set-valued optimization approach to explain bilevel programming problems. 

The structure of the paper is as follows. In Section \ref{prelim}, we recall some basic notions of bilevel programming and set-valued optimization. In Section \ref{scalar connection}, we consider the set-valued formulation of bilevel programming as proposed in \cite{Dempedutta2006bilevelwithconvexlower,Dutta2017optimalityofbilevelthroughvariational} and show how set-valued solutions are related to existing notions of solutions for bilevel programming problems. We conclude the paper in Section \ref{conclusion}. 

\section{Preliminaries}
\label{prelim}
We start by recalling some basics of bilevel programming. Bilevel programming, or in general multilevel programming, is concerned with the hierarchy of decisions (see \cite{Bard,Bard1988,Dempe2002foundation}). Typically, in bilevel programming, there are two levels of decision-making and two decision-makers, the leader and the follower, whose intentions are to optimally choose their decisions. The bilevel problem is the leader's problem. In this paper, we shall consider the following bilevel programming problem:
\begin{gather}\label{bilevel real upper}
    \underset{x}{\text{`minimize'}} \; F(x,y) \\ \text{subject to } x\in X, y\in \psi(x), \nonumber
\end{gather}
where, $\mathcal{X},\mathcal{Y}$ are two real normed linear spaces, $X\subseteq \mathcal{X}$ is a nonempty subset, $F: \mathcal{X} \times \mathcal{Y} \rightarrow \mathbb{R}$ is the leader's objective function, and for each $x\in X$,  $\psi(x):= \underset{y}{\arg\;\min} \lbrace f(x,y) \mid y\in K(x) \rbrace$ denotes the global solution set of the lower-level problem parametrized by $x$, which is given as 
 
\begin{gather}\label{bilevel lower}
    \underset{y}{\text{minimize}} \;f(x,y) \\ \text{subject to } y\in K(x) \nonumber,
\end{gather}
where for each $x\in X$, $K(x)$ is the constraint set of the lower-level player, which is influenced by the choice of the leader's decision $x$.
It must be mentioned that bilevel problems are very hard to solve, even in finite dimensions. First of all, the lower-level problem needs to be solved globally for each of the upper-level decisions $x$, which is a daunting task. Mostly, it is assumed that the function $y\mapsto f(x,y)$ is convex in $y$ for each $x\in X$. However, some nonconvex problems have also been considered in the literature, for example, the famous Mirrlees's problem (see \cite{mirrlees1999theory}).  As customary, we assume here that $\psi(x)$ is nonempty and fully computable for each $x\in X$; that is, the parametrized lower-level problem is fully globally solvable. Now if for each $x\in X$, $\psi(x)$ is singleton, say, $\psi(x)=\lbrace y(x) \rbrace$, then the upper-level problem reduces to the following problem:
\begin{gather}
    \underset{x}{\text{minimize}} \; F(x,y(x)) \\ \text{subject to } x\in X.  \nonumber
\end{gather}
However, unless the solution sets $\psi(x)$ of the lower-level problems are singleton, which is more common a case, problem (\ref{bilevel real upper}) is not well-defined, and that is why the quotation mark has been used in the initial formulation. There are many ways to avoid this ambiguity. Two such ways give rise to what is known as optimistic and pessimistic reformulation. In the optimistic scenario, the upper-level player expects cooperation from the lower-level player, and in this case, the bilevel problem/upper-level problem is reformulated as the following (see \cite{Dempe2002foundation}):
\begin{gather}\label{bilevel optimistic}
    \underset{x}{\text{minimize}} \; F_o(x) \\ \text{subject to } x\in X, \nonumber
\end{gather}
where $F_o(x)=\underset{y \in \psi(x)}{\inf}F(x,y)$.

However, cooperation from the lower-level player may not always model real-life situations, and the upper-level player may have to prepare for the worst case. This is the typical reasoning by game theorists, and this reasoning implies reformulating the upper-level problem for the leader as follows (see \cite{Leitmann}):
\begin{gather}\label{bilevel pessimistic}
    \underset{x}{\text{minimize}} \; F_p(x) \\ \text{subject to } x\in X, \nonumber
\end{gather}
where $F_p(x)=\underset{y \in \psi(x)}{\sup}F(x,y)$.

 Note that, in formulations of (\ref{bilevel optimistic}) and (\ref{bilevel pessimistic}), the functions $F_o$ and $F_p$ are infimal and supremal function values of the upper-level objective function over the solution set of the lower level parametrized problem. Here, it is not demanded that those values must be attained. But in many papers, this assumption is taken by default, and the infimum and supremum are replaced by minimum and maximum, respectively. However, here, we do not, in general, assume that the values are attained. Attainment will be specified when necessary, and we shall call those situations as value attainment situations. By value attainment, we shall mean both of the following situations: value attainment in the inner minimization in the optimistic formulation and value attainment in the inner maximization in the pessimistic formulation; and which one is being referred to will be clear from the context. Whereas pessimistic formulation is more common in the game theory literature, the optimistic formulation has interested mathematical optimizers because it is simpler to handle in comparison to the pessimistic formulation. In fact, more research has been devoted to optimistic formulation because of its close connection with the following (and relatively more tractable) formulation:
 \begin{gather}\label{bilevel standard upper}
    \underset{x,y}{\text{minimize}} \; F(x,y) \\ \text{subject to } x\in X, y\in \psi(x). \nonumber
\end{gather}
Note that in the above formulation, we need to minimize the upper-level objective with respect to both leaders' and followers' decisions. Though formulation (\ref{bilevel optimistic}) is closely related to formulation (\ref{bilevel standard upper}), they are not equivalent (though many articles claim to be so) as has been shown in \cite{Dempedutta2006bilevelwithconvexlower}. It is a widely popular belief that in terms of global solutions, both formulations are equivalent. However, that may not be the case unless the value in the inner minimization in the formulation (\ref{bilevel optimistic}) is attained.  It also should be noted that problems (\ref{bilevel optimistic}) and (\ref{bilevel standard upper}) are nonconvex problems because of the presence of the set $\psi(x)$, and hence we must focus largely on local solutions. It has been established that in the sense of local solutions, the two optimistic formulations are not equivalent. Problem formulation (\ref{bilevel optimistic}) appears more reasonable because, as an upper-level decision maker, the value only makes more sense than its attainment over the lower-level player's solution set. However, formulation (\ref{bilevel standard upper}) seems to be more tractable. Therefore, in the literature, formulation (\ref{bilevel optimistic}) is sometimes called a `real optimistic' upper-level problem, whereas formulation (\ref{bilevel standard upper}) is referred to as a `standard optimistic' upper-level problem. Let us now see some existing notions of solutions for bilevel programming and a few examples. 

\begin{definition}[\cite{Dempedutta2006bilevelwithconvexlower}]
A point $x^*\in X$ is called a global real optimistic solution to the problem (\ref{bilevel real upper}) if $F_o(x) \geq F_o(x^*)$ for all $x\in X$.
\end{definition}

\begin{definition}
A point $x^*\in X$ is called a local real optimistic solution to the problem (\ref{bilevel real upper}) if there exists an $\epsilon>0$ such that $F_o(x) \geq F_o(x^*)$ for all $x\in X$ with $\| x-x^* \|< \epsilon$.
\end{definition}

\begin{definition}
A point $(x^*,y^*)$ is called a global standard optimistic solution to the problem (\ref{bilevel real upper}) if $x^*\in X$, $y^*\in \psi(x^*)$ and  $F(x,y) \geq F(x^*,y^*)$ for all $x\in X$ and $(x,y)\in Graph(\psi)$.
\end{definition}

\begin{definition}\label{onedefinition}
A point $x^*\in X$ is called a local standard optimistic solution to the problem (\ref{bilevel real upper}) if there exists an $\epsilon>0$ such that $F(x,y) \geq F(x^*,y^*)$ for all $x\in X$ and $(x,y)\in Graph(\psi)$ with $\| (x,y)-(x^*,y^*)\|< \epsilon$.
\end{definition}
In the above definition, we have used a norm in the product space $\mathcal{X}\times \mathcal{Y}$, which we, by convention, consider as follows: $\|(x,y)\|_{\mathcal{X}\times \mathcal{Y}} = \|x\|_{\mathcal{X}} + \|y\|_{ \mathcal{Y}}$.
\begin{definition}
A point $x^*\in X$ is called a global real pessimistic solution to the problem (\ref{bilevel real upper}) if $F_p(x) \geq F_p(x^*)$ for all $x\in X$.
\end{definition}

\begin{definition}
A point $x^*\in X$ is called a local real pessimistic solution to the problem (\ref{bilevel real upper}) if there exists an $\epsilon>0$ such that $F_p(x) \geq F_p(x^*)$ for all $x\in X$ with $\| x-x^*\|< \epsilon$.
\end{definition}

Now, we provide some examples of bilevel programming. The first example is a special case of Stackelberg's game. To model market economy, Stackelberg first introduced the notion of the Stackelberg game in his monograph \cite{stackelberg1952theory}. Stackelberg's game is a special case of bilevel programming problems with the uniqueness of the solution in the lower-level problem. The notion of generalized Stackelberg game and generalized Stackelberg equilibrium has been introduced in \cite{Leitmann}. 
\begin{example}
Let there be two firms, say Firm 1 and Firm 2, who produce one identical product in the quantities $x_1$ and $x_2$, respectively. Assume that the unit production cost is simply $c$ for each firm. However, the price of the product in the market is determined by the total production, and let the unit price function be $p(x)=[\alpha-x]_+:=\max \{\alpha-x,0\}$. The value of $\alpha$ denotes the upper bound on the total production of two firms, beyond which the effective price of the product becomes $0$. Function $p(\cdot)$ is more commonly known as the inverse demand function. It is assumed that $\alpha>c$. Each firm's objective is to minimize the loss (or maximize the profit). However, Firm 1 acts as a leader and has the advantage that it first declares its production plan, and then only Firm 2 gets to decide its own production plan. Because of this hierarchy, Firm 2 solves the following problem:
   \begin{gather*}
        \underset{x_2}{\text{minimize}}\; f(x_1,x_2)=cx_2-x_2p(x_1+x_2)\\ \text{subject to } x_2\geq 0.  \nonumber
   \end{gather*}
   For this follower's problem, $x_1$ works as a parameter, which is supplied by the leader or Firm 1. The leader's problem then looks like
   \begin{align*}
        \underset{x_1}{\text{minimize}}\; &F(x_1,x_2)=cx_1-x_1p(x_1+x_2) \\
        \text{ subject to }&\text{ (i) } x_1\geq 0\\ &\text{ (ii) }x_2\in \underset{x_2}{\arg\;\min}\lbrace cx_2-x_2p(x_1+x_2) \mid x_2\geq 0 \rbrace.
   \end{align*}
   For the particular choice of the price function $p(x)=[\alpha-x]_+$, no one would produce more than $\alpha$ quantity, and that is an implicit constraint. The lower-level objective function takes the following form:
   \begin{gather*}
       f(x_1,x_2)=\begin{cases}
        (c-\alpha+x_1)x_2+x_2^2 & \text{if } x_2\leq \alpha-x_1\\
        cx_2 & \text{if } x_2> \alpha-x_1
    \end{cases}.
   \end{gather*}
This can be solved globally and achieves its minimum at a unique point depending on the values of $x_1$. The solution set map is as follows:
\begin{gather*}
       \psi(x_1)=\begin{cases}
        \left\lbrace \frac{(\alpha-c-x_1)}{2} \right\rbrace & \text{if } 0\leq x_1 \leq \alpha-c\\
        \lbrace 0 \rbrace & \text{if } x_1> \alpha-c
    \end{cases}.
   \end{gather*}
Plugging these values in the upper-level problem, the reduced upper-level problem becomes
   \begin{gather*}
        \underset{x_1}{\text{minimize}}\; F(x_1,\psi(x_1))=cx_1-x_1\left(\alpha-x_1-\frac{\alpha-c-x_1}{2}\right)\\ \text{subject to } x_1\geq 0.  \nonumber
   \end{gather*}
This problem can also be solved globally and has a unique global minimum at $(\frac{\alpha-c}{2})$. Thus $(\frac{\alpha-c}{2},\frac{\alpha-c}{4})$ is the global standard optimistic as well as $(\frac{\alpha-c}{2})$ is the global real optimistic and global real pessimistic solution of the bilevel problem. 
\end{example}

\begin{example}
The second example is of an extensive form game. There is a lot of literature on extensive form games; one may look at \cite[Ch.~6]{Osbornegametheorybook1994} for a quick primer. Extensive-form games are usually represented by game trees. There is a root vertex, which is the starting point of the game. Each of the terminal vertices represents an outcome of the game and consists of a tuple of payoffs (here, we will write costs and minimize them) for all the players. Each of the non-terminal vertices represents a game situation. The most interesting question about any game is on how to find the equilibrium profile of the game. For extensive form games, the most popular notion of equilibrium is that of a subgame perfect Nash equilibrium.  A strategy profile is called a subgame perfect Nash equilibrium if it is a Nash equilibrium of every subgame of the original game. We shall illustrate it for a two-player extensive form game of length two.

Consider an extensive form game of two players of length two. Let $X$ and $Y=\underset{x\in X}{\bigcup}Y(x)$ denote all possible moves that the first player and the second player can play, respectively. A pair $(x,f)$ where $x\in X$ and $f:X \rightarrow Y$ with $f(x)\in Y(x)$ is called a subgame perfect Nash equilibrium if no player wants to deviate given the other one stick to his/her choice. As an example, let the first player have only two strategies, say $A$ and $B$. When the first player plays $A$, the second player, suppose, has only two strategies, say $C$ and $D$. When the first player plays $B$, the second player, suppose again, has only two strategies, say $E$ and $F$. Then, the graph in Figure \ref{fig:game 1} represents the game tree of this extensive form game.
    
    \begin{figure}[ht]
        \centering
        \includegraphics{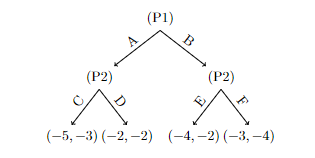}
        \caption{A game with a unique subgame perfect Nash equilibrium.}
        \label{fig:game 1}
    \end{figure}
    
Since we have assumed it to be a game of length two, the game ends here, and each terminal node corresponds to a pair of costs for both players. Each player wants to minimize his/her cost. Using the backward induction method, it can be seen that the strategy profile $\left(A, \begin{pmatrix} A\rightarrow C \\ B\rightarrow F \end{pmatrix}\right)$ is a subgame perfect Nash equilibrium. Interestingly, this decision problem can be formulated in the language of bilevel programming. Here, it can be seen that for each of the leader's decisions ($A$ and $B$), the follower uniquely achieves his/her minimum cost ($C$ for $A$ and $F$ for $B$). And out of $A$ and $B$, $A$ fetches the leader a lesser cost. In the language of bilevel programming, it means that $(A,C)$ is a global standard optimistic solution (as well as $A$ is a global real optimistic and global real pessimistic solution). Thus, the solutions of the bilevel formulation are connected to the notion of subgame perfect Nash equilibrium. In fact, when each player's strategy sets are finite, it can be shown that each of the global real optimistic (equivalently global standard optimistic because the value is attained because of finiteness) solutions and global real pessimistic solutions correspond to some subgame perfect Nash equilibrium. For example, consider the following extensive form game represented in Figure \ref{fig:game 2}.
\begin{figure}[!ht]
    \centering
    \includegraphics{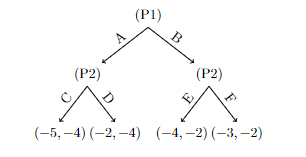}
    \caption{A game with multiple subgame perfect Nash equilibria.}
    \label{fig:game 2}
\end{figure}

Here again, we have assumed that the players have the same strategies as in the previous game, but with different costs. In this game, the second player's alternative choices incur the same cost to him/her for each of the first player's choices. But his/her choice affects the first player's costs. Here it can be seen that $\left(A, \begin{pmatrix} A\rightarrow C \\ B\rightarrow E \end{pmatrix}\right)$ and $\left(B, \begin{pmatrix} A\rightarrow D \\ B\rightarrow F \end{pmatrix}\right)$ are both subgame perfect Nash equilibria. Also, $(A,C)$ is a global standard optimistic solution, and $B$ is a global real pessimistic solution of the bilevel formulation. This shows that there is a close relation between subgame perfect equilibrium strategies and solutions of the bilevel formulation. But not all the subgame perfect Nash equilibrium may have a correspondence with some real optimistic or real pessimistic solution of the bilevel formulation. For example, consider the following game with the game tree given by:

\begin{figure}[ht]
    \centering
    \includegraphics{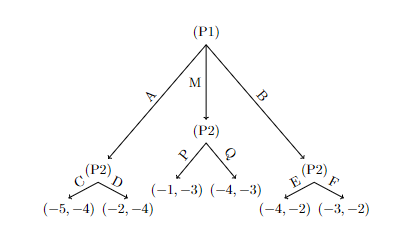}
    \caption{A game that has a subgame perfect Nash equilibrium that does not correspond to any real optimistic or pessimistic solutions.}
    \label{fig:game 3}
\end{figure}

In this case, it can be seen that though $\left(M, \begin{pmatrix} A\rightarrow D \\ B\rightarrow E \\M \rightarrow Q \end{pmatrix}\right)$ is a subgame perfect Nash equilibrium, $M$ does not correspond to any of the real optimistic or pessimistic solutions of the bilevel formulation. In fact, $A$ is the global real optimistic solution (correspondingly $(A,C)$ is a global standard optimistic solution), and $B$ is the global real pessimistic solution of the bilevel formulation.

When each player's strategy set becomes infinite, it can still be proved that every global standard optimistic solution will correspond to a subgame perfect Nash equilibrium. But for global real optimistic and real pessimistic solutions to correspond to some subgame perfect Nash equilibrium, one would need the value attainment assumption. For example, consider a two-person extensive form game of length two whose corresponding bilevel problem is given as follows:

 \begin{gather}
    \underset{x}{\text{`minimize'}} \; F(x,y)=x-y \\ \text{subject to } x\in [0,1], y\in \psi(x). \nonumber
\end{gather}

 \begin{gather}
    \underset{y}{\text{minimize}} \;f(x,y)=\begin{cases}
        \lfloor{y}\rfloor & \text{if } x=0\\
        -xy & \text{if } x >0
    \end{cases} \\ \text{subject to } y\in [0,1].  \nonumber
\end{gather}
Here, $\lfloor.\rfloor$ is the standard greatest integer function or the floor function. The lower-level solution set can be calculated to:
$$ \psi(x)=\begin{cases}
    [0,1) & \text{if } x=0\\
    \lbrace 1 \rbrace & \text{if } 0<x\leq 1
\end{cases}.$$
Optimistic upper-level objective function $F_o(x)$ turns out to be
$$ F_o(x)=x-1\;\;\text{for all }x\in[0,1].$$
Thus, for this problem, $x=0$ is a global real optimistic solution. However, $0$ does not correspond to any subgame perfect Nash equilibrium. In fact, it can be shown that this game does not have any subgame perfect Nash equilibrium.
\end{example}

\begin{example}
We now provide another example from the field of robust optimization. Robust optimization is one of the prominent ways to handle uncertain optimization problems. Consider the following uncertain scalar optimization problem (see \cite{BNT})
 \begin{gather} \label{scalarform}
  \underset{x}{\text{minimize }} \;\phi(x,\xi) \\ \text{subject to } x \in S, \nonumber
\end{gather}
where $\phi: \mathbb{R}^n \times U \rightarrow \mathbb{R}$ is a map, $S \subseteq \mathbb{R}^n$ is a known constraint/feasible set, and $U \subseteq \mathbb{R}^k$ is the set of uncertain scenarios. Here, we have assumed that the uncertainty parameter only appears in the objective. Problem (\ref{scalarform}) in itself is not a well-defined problem. When we fix one $\xi \in U$, this is a simple scalar optimization problem, and hence problem (\ref{scalarform}) can be viewed as a collection of scalar optimization problems $\lbrace P(\xi): \xi \in U \rbrace$, where, for each $\xi \in U$, problem $P(\xi)$ is given as:
\begin{equation}
    \tag{$P(\xi)$}
    \begin{gathered}
        \underset{x}{\text{minimize }} \;\phi(x,\xi)\\ \text{subject  to } x \in S.
    \end{gathered}
\end{equation}

There are multiple ways to define robust solutions for the problem (\ref{scalarform}) based on different ways of interpreting the uncertainty in the solution concept. The scenario-based interpretation gives rise to notions of \text{`highly'} and \text{`flimsily'} robust solutions. However, the notion that is most popular and most researched is the notion of \textit{`min-max robust'} solution (also known as {worst case robust solution} or {strict robust solution} or simply {robust} solution in the literature). The min-max robust solution is defined through the solution of the following so-called \text{robust counterpart}: 
\begin{gather} \label{scalarformcounterpart}
  \underset{x}{\text{minimize }}~ \left(\underset{\xi \in U}{\sup}\;\phi(x,\xi) \right) \\ \text{subject to } x \in S. \nonumber
  \end{gather}
  Note that the formulation (\ref{scalarformcounterpart}) is a single problem arising out of the family of problems $\lbrace P(\xi): \xi \in U \rbrace$. As the formulation suggests, it considers the worst case that can happen due to data uncertainty and then tries to minimize the worst-case objective value. In terms of the solution to this associated problem, the min-max robust solution for the problem (\ref{scalarform}) is defined as follows (see \cite{BNT}):
      A point $x^0 \in S$ is called a \textit{min-max} robust solution to the problem (\ref{scalarform}) if it is an optimal solution of the problem (\ref{scalarformcounterpart}), that is, \begin{center}
          
      $\underset{\xi \in U}{\sup}\;\phi(x^0,\xi) \leq \underset{\xi \in U}{\sup}\;\phi(x,\xi)$ for all $x \in S$.  \end{center}
  
While {min-max robustness} incorporates the pessimistic way to look at the uncertainty, the optimistic interpretation leads to the concept of {optimistic robustness} (see \cite{Beck,Klamroth17}). This also appears when one tries to formulate the dual of the min-max formulation (see \cite{Beck}). Corresponding to the problem (\ref{scalarform}), consider the following counterpart:
 \begin{gather} \label{scalaroptimisticformcounterpart}
  \underset{x}{\text{minimize }}~ \left(\underset{\xi \in U}{\inf}\;\phi(x,\xi) \right) \\ \text{subject to } x \in S. \nonumber
\end{gather}
A point $x^0 \in S$ is called an \textit{optimistic robust} solution to the problem (\ref{scalarform}) if it is an optimal solution of the problem (\ref{scalaroptimisticformcounterpart}), that is, \begin{center}
    $\underset{\xi \in U}{\inf}\;\phi(x^0,\xi) \leq \underset{\xi \in U}{\inf}\;\phi(x,\xi)$ for all $x \in S$. \end{center} 

What we show now is that we can think about these robust solutions as solutions to some bilevel programming problems. Consider the following bilevel optimization problem corresponding to the problem (\ref{scalarform}).

The lower-level problem  as:
\begin{gather}\label{bilevel lower corresponding to robust}
    \underset{\xi}{\text{minimize }} \;\hat{f}(x,\xi) \\ \text{subject to } \xi \in  U, \nonumber
\end{gather}
where \begin{center}
    
$\hat{f}(x,\xi)= \left\{
  \begin{array}{lll} 
      0 & & \text{ if }(x,\xi) \in S\times U \\
      1 & & \text{ otherwise} 
      \end{array}
\right.
$
\end{center}
and the upper-level problem :
\begin{gather}\label{bilevel optimistic corresponding to robust}
    \underset{x}{\text{minimize }} \; \phi(x,\xi) \\ \text{subject to } x\in S. \nonumber
\end{gather}

Note that in the above formulation, the lower-level problem is actually a dummy problem that returns the solution set $\psi(x)$ as $U$ for every $x\in X$. Also, this is a problem that does not have a unique lower-level solution for each of the upper-level player's choices.  Now, it can be seen that every optimistic robust solution to this problem is a global real optimistic solution to the bilevel problem (\ref{bilevel optimistic corresponding to robust}) and vice versa. Similarly, every min-max robust solution of problem (\ref{scalarform}) is a global real pessimistic solution to the bilevel problem (\ref{bilevel optimistic corresponding to robust}) and vice versa.

In case the inner supremum and infimum as in problems (\ref{scalarformcounterpart}) and (\ref{scalaroptimisticformcounterpart}), respectively, are attained for every $x\in S$ (which can happen in many ways, for example, if the function $\phi(x,\xi)$ is continuous in $\xi$ and the uncertainty set $U$ is compact), alternative bilevel problems can be formulated whose solution will correspond to min-max and optimistic robust solution. For the problem (\ref{scalarform}), and its min-max counterpart as the problem (\ref{scalarformcounterpart}), consider the following bilevel optimization problem.

The lower-level problem :
\begin{gather}\label{bilevel pessimistic lower for robust}
    \underset{\xi}{\text{minimize }} \;-{\phi}(x,\xi) \\ \text{subject to } \xi \in  U, \nonumber
\end{gather}

and the pessimistic upper-level problem :
\begin{gather}\label{bilevel pessimistic upper for robust}
    \underset{x}{\text{minimize }} \; \phi_p(x) \\ \text{subject to } x\in S, \nonumber
\end{gather}
where $\phi_p(x)=\underset{\xi \in U}{\sup}\;\phi(x,\xi)$. Note that we have taken the supremum over whole $U$ in the formation of $\phi_p$ instead of over the lower level solution set. However, since we are just interested in the value, they both are the same.

Similarly, for the problem (\ref{scalarform}), and its optimistic counterpart as the problem (\ref{scalaroptimisticformcounterpart}), consider the following bilevel optimization problem.

The lower-level problem :
\begin{gather}\label{bilevel optimistic lower for robust}
    \underset{\xi}{\text{minimize }} \;{\phi}(x,\xi) \\ \text{subject to } \xi \in  U, \nonumber
\end{gather}

and the optimistic upper-level problem :
\begin{gather}\label{bilevel optimistic upper for robust}
    \underset{x}{\text{minimize }} \; \phi_o(x) \\ \text{subject to } x\in S, \nonumber
\end{gather}
where $\phi_o(x)=\underset{\xi \in U}{\inf}\;\phi(x,\xi)$.

Then, we can have the following result whose proof readily follows from the constructions of the corresponding problems. We sketch part of the proof and leave it to the reader to complete.
\begin{lemma}\label{th2}
Consider the robust optimization problem (\ref{scalarform}) where the inner supremum and infimum are attained in the min-max and optimistic robust counterparts (\ref{scalarformcounterpart}) and (\ref{scalaroptimisticformcounterpart}), respectively. Then every min-max robust solution of this problem is a global real pessimistic solution to the bilevel problem (\ref{bilevel pessimistic lower for robust})-(\ref{bilevel pessimistic upper for robust}) and vice versa. Similarly, every optimistic robust solution of problem (\ref{scalarform}) is a global real optimistic solution to the bilevel problem (\ref{bilevel optimistic lower for robust})-(\ref{bilevel optimistic upper for robust}) and vice versa.
\end{lemma}

\begin{proof}
    We give proof of one of the equivalence, namely the equivalence between the min-max robust solution and the global real pessimistic solution of the bilevel formulation. The other equivalence can be similarly derived. If one sees carefully, the assumption of value attainment of the inner supremum and infimum in the min-max and optimistic robust counterparts (\ref{scalarformcounterpart}) and (\ref{scalaroptimisticformcounterpart}), respectively is taken to ensure that the lower-level solution sets of the bilevel problems (\ref{bilevel pessimistic lower for robust})-(\ref{bilevel pessimistic upper for robust}) and (\ref{bilevel optimistic lower for robust})-(\ref{bilevel optimistic upper for robust}) are nonempty for all upper-level decision $x$. Also note that for any $x\in S$, the solution set of the lower-level problem (\ref{bilevel pessimistic lower for robust}) is given by \begin{align*}
        \psi_1(x) & =\{\xi \in U\;:\; t \mapsto -\phi(x,t) \text{ is minimized at }\xi\}\\
        & =\{\xi \in U\;:\; t \mapsto \phi(x,t) \text{ is maximized at }\xi\}.
    \end{align*}
    We also know that by assumption $\psi_1(x)$ is nonempty for all $x\in S$. Thus
    \begin{align*}
        \phi_p(x) &=\underset{\xi \in U}{\sup}\;\phi(x,\xi)\\
        &=\phi(x,t) \; \text{for all }t\in \psi_1(x)\\
        &=\underset{t \in \psi_1(x)}{\sup}\;\phi(x,t).
    \end{align*}

    Now let us come to the main proof. Assume that $x^*\in S$ be a min-max robust solution for (\ref{scalarform}). Thus it is a solution for (\ref{scalarformcounterpart}). Hence
    \begin{align*} \underset{\xi \in U}{\sup}\;\phi(x^*,\xi) &\leq \underset{\xi \in U}{\sup}\;\phi(x,\xi)\; \text{for all }x\in S, \\
    \text{that is, }\;\;\; \phi_p(x^*) & \leq \phi_p(x)\; \text{for all }x\in S.
    \end{align*}
    Hence $x^*$ is a global real pessimistic solution of the bilevel problem (\ref{bilevel pessimistic lower for robust})-(\ref{bilevel pessimistic upper for robust}).

    Conversely, if $y^*$ is a global real pessimistic solution of the bilevel problem (\ref{bilevel pessimistic lower for robust})-(\ref{bilevel pessimistic upper for robust}), then 
    \begin{align*} \phi_p(y^*) & \leq \phi_p(y)\; \text{for all }y\in S,\\ \text{that is, }\underset{\xi \in U}{\sup}\;\phi(y^*,\xi) &\leq \underset{\xi \in U}{\sup}\;\phi(y,\xi)\; \text{for all }y\in S .
    \end{align*} 
    Thus, $y^*$ is a min-max robust solution for (\ref{scalarform}).
\end{proof}

\end{example}

\begin{example}
Now, we give an example where infinite dimensional space actually occurs. The problem is known as the obstacle problem in variational inequalities literature. We take it from \cite{mehlitz2017contributions}. Consider a bounded domain $\Omega\subseteq \mathbb{R}^d$ with a sufficiently smooth boundary. Let $X\subseteq L^2(\Omega)$ be a nonempty closed convex set, $y_d\in L^2(\Omega)$ be a desired state, and $\sigma >0$ be a regularization parameter. The obstacle problem is the following hierarchical problem:
    
    \begin{gather}\label{obstacle upper}
    \underset{x,y}{\text{minimize }} \frac{1}{2}\lvert \lvert y-y_d\rvert \rvert_{L^2(\Omega)}^2+\frac{\sigma}{2}\lvert \lvert x\rvert \rvert^2_{L^2(\Omega)} \\ \text{subject to } x\in X, \nonumber \\y \in \Psi(x), \nonumber
\end{gather}
 where $\Psi(x)$ denotes the solution set mapping of the lower-level problem:
 
 \begin{gather}
    \underset{y}{\text{minimize }} \frac{1}{2}\int_{\Omega}\lvert \triangledown y(\omega)\rvert^2_2 d\omega-\int_{\Omega}x(\omega)y(\omega)d\omega \\ \text{subject to } y\in H^1_0(\Omega), \nonumber \\y(w)\geq 0 \text{ almost everywhere on }  \Omega . \nonumber
\end{gather}
Here $H^1_0(\Omega)$ is the Sobolev space consisting of closure of the infinitely differentiable functions compactly supported in $\Omega$ in $H_1(\Omega)$ with the standard Sobolev norm. 

In the language of bilevel programming, it is a standard optimistic formulation of a bilevel decision problem as in the upper-level problem (\ref{obstacle upper}), one is trying to minimize over $x$ as well as $y$. 
\end{example} 

We have now given sufficient number of examples to illustrate that bilevel problems do occur in many practical decision-making situations. We now move to the theory of bilevel programming and see why set-valued optimization is of relevance. In the paper \cite{Dempedutta2006bilevelwithconvexlower} (also see \cite{Dutta2017optimalityofbilevelthroughvariational}), the authors have shown that when the lower-level solution set $\psi(x)$ is not singleton for each $x\in X$, the upper-level problem actually takes the following set-valued form:
\begin{gather}\label{bilevel set-valued upper}
    \underset{x}{\text{minimize}} \; F(x,\psi(x)) \\ \text{subject to } x\in X \nonumber ,
\end{gather}
where $F(x,\psi(x)):=\lbrace F(x,y) \mid y\in \psi(x) \rbrace$ is a set for each $x\in X$.

Thus, set-valued optimization occurs very naturally in bilevel programming. This set-valued formulation looks more natural because for the upper-level decision maker, for each of his/her decision $x$, the set $F(x,\psi(x))$ contains all the possible outcomes that can happen when the lower-level decision maker or the follower chooses his/her decision optimally. This is important since the leader does not know beforehand which of the optimal decisions the follower would choose. This set-valued formulation has been studied in \cite{Zemkohoillposed,pileckathesis}. However, in those papers, the value attainment assumption has been imposed. In \cite{Zemkohoillposed}, the optimistic and pessimistic formulations are in terms of minimum and maximum instead of infimum and supremum. In \cite{pileckathesis}, the functions $F$ and $f$ are assumed to be continuous, and $\psi(x)$ has been assumed to be compact. In this paper, we are interested in the general case; we study the connection of the set-valued formulation with the optimistic and the pessimistic formulations without value attainment assumption. We first recall some basics of set-valued optimization before proving our main results.

A set-valued optimization problem in the most general form looks like: 
\begin{gather} \label{setvaluedoptimization}
    \text{minimize } G(x) \\ \text{subject to } x \in S, \nonumber
\end{gather} 
where $G:X \rightarrow 2^{Z}$ is a map, $X$ and $Z$ are normed linear spaces, and $S\subseteq X$ a nonempty constraint set. The power set of $Z$ is denoted by $2^Z$. Well-definedness of $G$ is preserved by considering the range as the power set of $Z$. Since optimization needs order structure, it is assumed that 
$Z$ is partially ordered by a nonempty closed convex pointed cone $C$, where the induced order $\leq_C$ is defined as: for $z_1,z_2\in Z$, $z_1\leq_C z_2$ if and only if $z_2-z_1\in C$.

Early literature on optimization with set-valued maps can be traced back to the works of Borwein (see \cite{borwein1977multivalued,borwein1977proper,borwein1981lagrange}), Postolic\u{a} (see \cite{postolm1986vectorial}), Corley (see \cite{corley1987existence,corley1988optimality}). A comprehensive study can be found in the book by Luc \cite{Luc89}. The notion of solution that has been considered in these papers is later called the `vector approach' of solution as it is a generalization of the solution notion of vector optimization problems (see \cite{Luc89}).
 \begin{definition}
 Consider the problem {\normalfont(\ref{setvaluedoptimization})}. The tuple $(x^0,y^0) \in S\times Z$ is called a \textit{global vector solution} to {\normalfont(\ref{setvaluedoptimization})} if $y^0 \in G(x^0)$ and $y^0$ is a minimal point of $G(S)=\underset{x\in S} {\bigcup} G(x)$, where the minimality is with respect to the order relation $\leq_C$.
 \end{definition}
By using the vector approach we find the minimal elements of the total image set $G(S)$ with respect to the underlying vector order $\leq_C$. Though this is a popular notion of solution, a drawback of this notion is that it only cares about one good point in the image set of a solution. In real life, it sometimes may not be quite the correct interpretation. For example, in a soccer league, one team with one excellent player and all other below-average players may not qualify as a strong team. Therefore, a better way of comparison among sets seemed more appropriate, and the so-called `set-relation' approach in set-valued optimization problems was introduced and popularized by Kuroiwa et al. (see \cite{KurTanTru97,DK1,DK,DK4}). The set-relation approach opened up a whole new area of research.
 
In the set-relation approach, sets are compared using set-order relations. For two nonempty subsets $A,B \subseteq Z$, consider the following set order relations:
\begin{itemize} 
	\item $A \leq ^l _C B $ if and only if $A+C \supseteq B$.
	\item $A \leq ^u _C B $ if and only if $B-C \supseteq A$.
\end{itemize} 
These set order relations are preorders (that is reflexive and transitive), and with respect to each such order, the set $2^Z$ is a preordered space. Of course, this is a disadvantage because these new orders are not partial orders anymore. 

 Let $\mathcal{A} \subseteq 2^Z$ be a nonempty collection of nonempty sets. We say that $A\in \mathcal{A}$ is an \textit{$l$-minimal} element of $\mathcal{A}$ if it is minimal in $\mathcal{A}$ with respect to the set order relation $\leq ^l _C$, that is, for any $B \in \mathcal{A} $, $B\leq_C^l A$ implies $A\leq_C^l B$. Similarly, we say that $A\in \mathcal{A}$ is an \textit{$u$-minimal} element of $\mathcal{A}$ if it is minimal in $\mathcal{A}$ with respect to the $\leq ^u _C$ set order relation, that is, for any $B \in \mathcal{A} $, $B\leq_C^u A$ implies $A\leq_C^u B$. 
 
 Based on the set order relations introduced above, the followings are a few notions of the `set-relation approach' of solutions for (\ref{setvaluedoptimization}) as has been introduced in \cite{DK1}.
\begin{definition}
Consider the problem {\normalfont(\ref{setvaluedoptimization})}. A point $x^0 \in S$ is called 
\begin{itemize}
\item[(i)] a \textit{global $l$-minimal} (also called \textit{$l$-type}) solution to {\normalfont(\ref{setvaluedoptimization})} if for any $x \in S$ such that $G(x) \leq^l_C$ $G(x^0)$ we have $G(x^0) \leq^l_C$ $G(x)$.

\item[(ii)] a \textit{global $u$-minimal} (also called \textit{$u$-type}) solution to {\normalfont(\ref{setvaluedoptimization})} if for any $x \in S$ such that $G(x) \leq^u_C$ $G(x^0)$ we have $G(x^0) \leq^u_C$ $G(x)$.
\end{itemize}
\end{definition}

The above definitions of solutions for set-valued optimization problems are that of global solutions. However, when dealing with nonconvex problems, which is the case for bilevel programming problems, local solutions are a more accurate entity to look for. Therefore, we also need notions of local solutions for set-valued optimization problems. With this aim, we define the following notions of local solutions for set-valued optimization problems.

\begin{definition}
Consider the problem {\normalfont(\ref{setvaluedoptimization})}.  
\begin{itemize}
\item[(i)] The tuple $(x^0,y^0) \in S\times Z$ is called a \textit{local vector solution} to {\normalfont(\ref{setvaluedoptimization})} if $y^0 \in G(x^0)$ and there exists an $\epsilon >0$ such that $y^0$ is a minimal point of $G\left(S\cap B(x^0,\epsilon) \right)$. 

\item[(ii)] A point $x^0 \in S$ is called a \textit{local $l$-minimal} solution to {\normalfont(\ref{setvaluedoptimization})} if there exists an $\epsilon >0$ such that for any $x \in S\cap B(x^0,\epsilon)$ with $G(x) \leq^l_C$ $G(x^0)$, we have $G(x^0) \leq^l_C$ $G(x)$.

\item[(iii)] A point $x^0 \in S$ is called a \textit{local $u$-minimal} solution to {\normalfont(\ref{setvaluedoptimization})} if there exists an $\epsilon >0$ such that for any $x \in S\cap B(x^0,\epsilon)$ with $G(x) \leq^u_C$ $G(x^0)$, we have $G(x^0) \leq^u_C$ $G(x)$.
\end{itemize}
\end{definition} 

We now recall a few terminologies for a set-valued map from the literature, which will be used for various purposes in this paper (For details one may look at \cite{Gofert}).
\begin{definition}
The set-valued map $G$ in (\ref{setvaluedoptimization}) is called
 
\begin{enumerate}[(i)]

\item[(i)] \textit{closed (compact)-valued} on $S$ if, for each $x\in S$, $G(x)$ is closed (compact, respectively). 


\item[(ii)] graph closed if graph of $G$, defined by $Graph(G):=\lbrace (x,y)\in X \times Z \mid y \in G(x)\rbrace$, is a closed subset of $X \times Z$.

\item[(iii)] locally bounded at $x_0\in S$ if there is a neighbourhood $ U$ of $x_0$ such that $G(U\cap S):=\underset{x\in U\cap S }{\bigcup}G(x)$ is bounded in $Z$.

\item[(iv)] \textit{upper continuous} at $x_0 \in S$ if for every open set $V$ in $Z$ with $G(x_0) \subseteq V$, there exists a neighbourhood $ U$ of $x_0$ such that $G(x) \subseteq V$ for every $x\in U \cap S$. $G$ is called upper continuous on $S$ if it is so at each point of $S$. 

\item[(v)] \textit{lower continuous} at $x_0 \in S$ if for every open set $V$ in $Z$ with $G(x_0) \cap V \neq \emptyset$, there exists a neighbourhood $U$ of $x_0$ such that $G(x) \cap V \neq \emptyset$ for every $x\in U \cap S$. $G$ is called lower continuous on $S$ if it is so at each point of $S$. 

\item[(vi)] 
\textit{continuous} at $x_0 \in S$ if it is both upper continuous and lower continuous at $x_0 \in S$. $G$ is called continuous on $S$ if it is so at each point of $S$.

\item[(vii)] 
\textit{compact} at $x_0 \in S$ if for every sequence $\lbrace x_n \rbrace_n \subseteq S$ with $x_n \rightarrow x_0$ and $y_n \in G(x_n)$, there exist a subsequence $\lbrace y_{n_k} \rbrace_k$ and an element $y_0 \in G(x_0)$ such that $y_{n_k} \rightarrow y_0$. $G$ is called compact on $S$ if it is so at each point of $S$.

\end{enumerate}

\end{definition}

\begin{remark}
The notions of upper continuity and lower continuity also appear under the name of upper semicontinuity and lower semicontinuity, respectively, in the literature. We followed these terminologies from \cite{Gofert}. When $G$ is single-valued, the above notions of upper and lower continuity coincide with the notion of continuity of single-valued maps, and hence the mentioned names look more appropriate. Also, a set-valued map is compact if and only if it is upper continuous and compact valued. In the famous book of Berge (\cite{berge1997topological}), the upper continuous map has always been assumed to be compact-valued. But that need not be the case here. For example, $G:\mathbb{R}\rightarrow 2^{\mathbb{R}}$ defined by $G(x)=[0,\infty)$ is upper continuous but not compact in our notations. 
\end{remark}

\section{Connection between existing notions of solutions of bilevel programming problem and solutions of the set-valued formulation}\label{scalar connection}

In this section, we focus on various reformulations of a bilevel programming problem and explore their connection with one another. It must be noted that global (local) real optimistic and global (local) real pessimistic solutions to the problem (\ref{bilevel real upper}) correspond to the global (local) solutions of (\ref{bilevel optimistic}) and (\ref{bilevel pessimistic}), respectively. The following relation among real optimistic solutions and standard optimistic solutions of  (\ref{bilevel real upper}) is established in the literature. 

\begin{proposition}\label{real to standard local global}
Let $x^*$ be a local (global) real optimistic solution to the problem (\ref{bilevel real upper}). Let there exist $y^*\in\psi(x^*)$ such that $F(x^*,y^*)=F_o(x^*)$. Then $(x^*,y^*)$ is a local (global) standard optimistic solution to the problem (\ref{bilevel real upper}).
\end{proposition}

\begin{proof}
The proof for the local solution is given in \cite[Proposition 3.1]{Dutta2017optimalityofbilevelthroughvariational}. The proof for the global solution is much simpler and can be left as an exercise to the reader. 
\end{proof}

The above proposition illustrates that the value attainment in the inner minimization is indeed necessary to go from a real optimistic solution to a standard optimistic solution. This value attainment assumption may not be satisfied even for nice bilevel problems with convex lower-level objective functions. By a convex lower-level problem, we mean that the lower-level objective is convex in the followers variable. We illustrate this through an example below.

\begin{example}
    Consider the following bilevel problem:
\begin{gather}
    \underset{x}{\text{`minimize'}} \; F(x,y)=e^{x-y} \\ \text{subject to } x\in [0,1], y\in \psi(x), \nonumber
\end{gather}
where $\psi(x)$ is the global solution set of the lower-level problem:
 \begin{gather}
    \underset{y}{\text{minimize}} \;f(x,y)=\max \lbrace x-y , 0 \rbrace 
\end{gather}
In this case, the lower-level solution set is $\psi(x)=[x,\infty)$ for every upper-level decision $x\in [0,1]$ and the optimistic value function $$F_o(x)=\underset{y\in[x,\infty)}{\inf} e^{x-y}= 0\;\; \text{ for all } x\in [0,1].$$ This problem does not have a global standard optimistic solution, though it has many global real optimistic solutions. In fact, all the $x\in [0,1]$ are global real optimistic solutions.
\end{example}

The above example illustrates an important point. Though a real optimistic solution appears to be a natural construct, the problem with a real optimistic solution is that, by making this choice, the upper-level player cannot guarantee to achieve the optimistic value $F_o$, unless the value attainment happens in the inner minimization. He/she can only ensure that under the value attainment assumption, when it also becomes a standard optimistic solution. 
 
Now, let us look at the reverse implication. With respect to the global solution, the following result holds.
\begin{proposition}{\normalfont(\cite[Proposition 3.2]{Dutta2017optimalityofbilevelthroughvariational})} \label{existing global standard to global real}
Let $(x^*,y^*)$ be a global standard optimistic solution to the problem (\ref{bilevel real upper}). Then $x^*$ is a global real optimistic solution to the problem (\ref{bilevel real upper}).
\end{proposition}

The above propositions show that globally, both real optimistic and standard optimistic formulations are well related, but they are not exactly equivalent because the value attainment of the inner minimization problem plays a role. The local solution case is even more difficult. While one can go from a local real optimistic solution to a local standard optimistic solution under the value attainment assumption, the other direction is not true in general, as can be seen from Example 1.3 in the monograph \cite{dempe2015book}, which we mention below.

\begin{example}\label{example local standard but not local real}
Consider the following bilevel programming problem where the lower-level problem is given by 
\begin{gather*}
    \underset{y}{\text{minimize}} \;f(x,y)=xy \\ \text{subject to } y\in [0,1] \text{ and } x,y \in \mathbb{R}, \nonumber
\end{gather*}

and the upper-level problem is 
\begin{gather*}
    \underset{x}{\text{`minimize' }} \; F(x,y)=x \\ \text{subject to } x\in [-1,1], y\in \psi(x). \nonumber
\end{gather*}
 The lower-level solution set map $\psi$ can be calculated to be
  $$ \psi(x)=\begin{cases}
      \{0\} & \text{ if } 0<x\leq 1\\
      \{1\} & \text{ if } -1\leq x<0\\
      [0,1] & \text{ if } x=0\\
  \end{cases}.$$
Also, since the upper-level objective function does not depend on the lower-level player's variable $y$, the optimistic value function $F_o$ is simply $F$, that is, $F_o(x)=F(x)$ for all $x\in [-1,1]$. The point $x=0$ is clearly not a local minimizer of $F_o$; that is, $x=0$ is not a local real optimistic solution. However, it can be seen that $(0,0)$ is a local standard optimistic solution by considering the ball of radius $\frac{1}{2}$ around the point $(0,0)$ and taking its intersection with the graph of $\psi$.
\end{example}

The following proposition from \cite{Dutta2017optimalityofbilevelthroughvariational} is one attempt to provide a sufficient condition to go from a local standard optimistic solution to a local real optimistic solution.

\begin{proposition}{\normalfont(\cite[Proposition 3.3]{Dutta2017optimalityofbilevelthroughvariational})}\label{existing local standard to local real}
Consider the bilevel programming problem (\ref{bilevel real upper}) where $\mathcal{Y}=\mathbb{R}^p$ for some $p\in \mathbb{N}$. Let $(x^*,y^*)$ be a local standard optimistic solution to the problem (\ref{bilevel real upper}) and assume that the lower level solution set map $\psi$ is locally bounded and graph closed. Then $x^*$ is a local real optimistic solution to the problem (\ref{bilevel real upper}) provided $\psi(x^*)=\lbrace y^*\rbrace$. 
\end{proposition}

In the next result, we show that the single-valuedness assumption
in the above proposition may be dropped.

\begin{proposition}\label{standard to real local finite dimension}
Consider the bilevel programming problem (\ref{bilevel real upper}) where $\mathcal{Y}=\mathbb{R}^p$ for some $p\in \mathbb{N}$. Let $(x^*,y^*)$ be a local standard optimistic solution to the problem (\ref{bilevel real upper}) and assume that the lower level solution set map $\psi$ is locally bounded and graph closed and $F$ is jointly lower semicontinuous in $x,y$. Then $x^*$ is a local real optimistic solution to the problem (\ref{bilevel real upper}) provided $y^*$ is the strict minimizer of $F(x^*,y)$ over $y \in \psi(x^*)$, that is, $F(x^*,y^*)<F(x^*,y)$ for all $y^*\neq y \in \psi(x^*)$.
\end{proposition}

\begin{proof}
It is given that $(x^*,y^*)$ is a local standard optimistic solution to the problem (\ref{bilevel real upper}). Also, by assumption, $F_o(x^*)=F(x^*,y^*)$. Suppose, if possible, that $x^*$ is not a local real optimistic solution to the problem (\ref{bilevel real upper}). Then we can construct a sequence $\lbrace x_n\rbrace_n$ in $X$ such that $x_n\rightarrow x^*$ and $F_o(x_n)<F_o(x^*)$ for all $n\in\mathbb{N}$. Now by definition of $F_o$ being infimum, we can find a sequence $\lbrace y_n \rbrace_n$  such that $y_n \in \psi(x_n)$ and $$F_o(x_n)\leq F(x_n,y_n)<F_o(x^*) \text{ for all } n\in \mathbb{N}.$$ Now, since $\psi$ is assumed to be locally bounded, $\lbrace y_n \rbrace_n$ is a bounded sequence and hence, using Bolzano-Weierstrass theorem, without loss of generality we may assume that $y_n\rightarrow \hat{y}$ for some $\hat{y}\in \mathcal{Y}$. Since the map $\psi$ is also assumed to be graph-closed, $\hat{y}\in \psi(x^*)$.  Now $(x_n,y_n)\rightarrow (x^*,\hat{y})$ and $F(x_n,y_n)<F_o(x^*)$. Since $F$ is lower semicontinuous, we get $F(x^*,\hat{y})\leq F_o(x^*)$ and hence by the definition of infimum, $F(x^*,\hat{y})= F_o(x^*)$. Now, two cases can arise. Either $\hat{y}\neq y^*$ or $\hat{y}= y^*$. If $\hat{y}\neq y^*$, we get a contradiction since $F(x^*,y^*)=F_o(x^*)<F(x^*,y)$ for all $y^*\neq y \in \psi(x^*)$ and $\hat{y}\in \psi(x^*)$. If $\hat{y}= y^*$, by construction $(x_n,y_n)\rightarrow (x^*,\hat{y})=(x^*,y^*)$ and $F(x_n,y_n)<F_o(x^*)=F(x^*,y^*)$. But this contradicts that $(x^*,y^*)$ is a local standard optimistic solution to the problem (\ref{bilevel real upper}). Therefore, our assumption that  $x^*$ is not a local real optimistic solution cannot be true. So $x^*$ is a local real optimistic solution to the problem (\ref{bilevel real upper}).
\end{proof}

The above result is true when $\mathcal{Y}$ is assumed to be finite-dimensional. However, the graph closed, and the local boundedness assumptions would not suffice the same when $\mathcal{Y}$ is an infinite-dimensional normed linear space. The following proposition can be derived for general normed linear spaces. 

\begin{proposition}\label{standard to real local}
Let $(x^*,y^*)$ be a local standard optimistic solution to the problem (\ref{bilevel real upper}) and assume that the lower level solution set map $\psi$ is compact and $F$ is jointly lower semicontinuous in $x,y$. Then $x^*$ is a local real optimistic solution to the problem (\ref{bilevel real upper}) provided $y^*$ is the strict minimizer of $F(x^*,y)$ over $y \in \psi(x^*)$.
\end{proposition}

\begin{proof}
It is given that $(x^*,y^*)$ is a local standard optimistic solution to the problem (\ref{bilevel real upper}). Also, by assumption, $F_o(x^*)=F(x^*,y^*)$. Suppose, if possible, that $x^*$ is not a local real optimistic solution to the problem (\ref{bilevel real upper}). Then we can construct a sequence $\lbrace x_n\rbrace_n$ in $X$ such that $F_o(x_n)<F_o(x^*)$ for all $n\in\mathbb{N}$. Now by definition of $F_o$, for each $n\in\mathbb{N}$ we can find an $y_n\in \psi(x_n)$ such that  $F_o(x_n)\leq F(x_n,y_n)<F_o(x^*)$. Now, since $\psi$ is assumed to be compact, without loss of generality, assume that $y_n\rightarrow \hat{y}$ for some $\hat{y}\in \psi(x^*)$. Now $(x_n,y_n)\rightarrow (x^*,\hat{y})$ and $F(x_n,y_n)<F_o(x^*)$. Since $F$ is lower semicontinuous, we get $F(x^*,\hat{y})\leq F_o(x^*)$. If $\hat{y}\neq y^*$, we get a contradiction since $F(x^*,y^*)<F(x^*,y)$ for all $y^*\neq y \in \psi(x^*)$ and $\hat{y}\in \psi(x^*)$. If $\hat{y}= y^*$, using the same argument as in the proof of the last proposition, we again get a contradiction that $(x^*,y^*)$ is a local standard optimistic solution to the problem (\ref{bilevel real upper}). Therefore, our assumption cannot be true. So $x^*$ is a local real optimistic solution to the problem (\ref{bilevel real upper}).
\end{proof}

We must mention that the strict minimizer assumption in Propositions \ref{standard to real local finite dimension} and \ref{standard to real local} is only a sufficient condition, as can be seen from the following example.

\begin{example}
Consider the lower-level problem given by 
\begin{gather*}
    \underset{y}{\text{minimize}} \;f(x,y)=xy \\ \text{subject to } y\in [0,1] \text{ and } x,y \in \mathbb{R}, \nonumber
\end{gather*}

and the upper-level problem is given by 
\begin{gather*}
    \underset{x}{\text{`minimize'}} \; F(x,y)=\lvert x \rvert \\ \text{subject to } x\in [-1,1], y\in \psi(x). \nonumber
\end{gather*}
Here again, the lower-level solution set map $\psi$ turns out to be
  $$ \psi(x)=\begin{cases}
      \{0\} & \text{ if } 0<x\leq 1\\
      \{1\} & \text{ if } -1\leq x<0\\
      [0,1] & \text{ if } x=0\\
  \end{cases}.$$
Also, since $F$ does not depend on $y$, $F_o(x)=F(x)$ for all $x\in [-1,1]$.
Here, it can be seen that $(0,0)$ is a local standard optimistic solution, and $0$ is a local real optimistic solution as well. However, because of nondependency of $F$ on $y$, $F_o(0)=0=F(0,y)$ for all $y\in[0,1]=\psi(0)$. This shows that the uniqueness assumption in Proposition \ref{standard to real local} is only sufficient and not necessary. 
\end{example}

It also must be mentioned that the condition that the solution map $\psi$ is compact in Proposition \ref{standard to real local} can be achieved under reasonable assumptions, as we show below. First, we recall the famous Berge's Maximum theorem, which we quote below without proof from \cite{Ausubel1993generalizedmaximumtheorem}, but in the context of normed linear spaces.

\begin{theorem}
Let $X$ and $Y$ be two normed linear spaces. Let $h:X\times Y\rightarrow \mathbb{R}$ be a continuous function and let $\gamma:X\rightarrow 2^Y$ be a continuous set-valued map that is nonempty and compact-valued. Then
\begin{itemize}
    \item[(i)] the function $m:X\rightarrow \mathbb{R}$ defined by $m(x):=\max \lbrace h(x,y) \mid y\in \gamma(x)\rbrace$ is continuous.
    \item[(ii)] the set-valued map $M:X\rightarrow 2^Y$ defined by $M(x):=\lbrace y\in Y \mid h(x,y)=m(x)\rbrace$ is upper continuous, nonempty and compact-valued.  
\end{itemize}
\end{theorem}

\begin{theorem}
Consider the bilevel programming problem (\ref{bilevel real upper}). Assume that the lower-level objective function $f$ is jointly continuous in $(x,y)$ and the set-valued constraint set mapping $x\mapsto K(x)$ is continuous, nonempty, and compact-valued. Assume the upper-level function $F$ is lower semicontinuous in $(x,y)$. Then, standard and real optimistic formulations are equivalent (locally and globally) in the sense that one's solutions correspond to another.
\end{theorem}

\begin{proof}
For $x\in X$, define $m(x)= \max \lbrace -f(x,y) \mid y\in K(x)\rbrace$ and $n(x)= \min \lbrace f(x,y) \mid y\in K(x)\rbrace$. Then $m(x)=-n(x)$. Recall that $\psi(x)$ is the solution set of the lower-level problem, that is, $\psi(x)=\lbrace y\in K(x) \mid f(x,y)=n(x) \rbrace$ for all $x\in X$. Define $\xi(x)=\lbrace y\in K(x) \mid -f(x,y)=m(x) \rbrace$. Then $\xi(x)=\psi(x)$. Now $f$ continuous means, $-f$ is also continuous, and hence, from Berge's maximum theorem, $\psi$ is upper continuous, nonempty, and compact-valued; that is, it is a compact map. Then, from Propositions \ref{real to standard local global}, \ref{existing global standard to global real}, and \ref{standard to real local}, we see that the solutions of standard and real optimistic formulations correspond to one another. 
\end{proof}

 Now, we start exploring how different existing notions of solutions to bilevel programming problems are related to the solutions of the set-valued formulation (\ref{bilevel set-valued upper}). Few such results can be found in \cite{pileckathesis,Zemkohoillposed} under the value attainment assumption. In particular, the standard optimistic formulation has not been considered in great detail in the above papers. We consider that case here and start with the connection between the standard optimistic solution and the solution of the set-valued formulation.

\begin{proposition}\label{global standard to global l}
Consider the usual ordering on $\mathbb{R}$ induced by the ordering cone $C=\mathbb{R}_+$. Let $(x^*,y^*)$ be a global standard optimistic solution to the problem (\ref{bilevel real upper}). Then $x^*$ is a global $l$-minimal solution and $(x^*,F(x^*,y^*))$ is a global vector solution of (\ref{bilevel set-valued upper}).
\end{proposition}

\begin{proof}
Since  $(x^*,y^*)$ is a global standard optimistic solution to the problem (\ref{bilevel real upper}), we have $F(x^*,y^*)\leq F(x,y)$ for all $x\in X$ and $(x,y)\in Graph(\psi)$. This implies $F(x^*,\psi(x^*))\leq_C^l F(x,\psi(x))$ for all $x\in X$. Therefore $x^*$ is a global $l$-minimal solution of (\ref{bilevel set-valued upper}). Also we get $F(x^*,y^*)\in F(x^*,\psi(x^*))$ is a minimum point of $\underset{x
\in X}{ \bigcup }F(x,\psi(x))$. Therefore $(x^*,F(x^*,y^*))$ is a global vector solution of (\ref{bilevel set-valued upper}).
\end{proof}

The situation becomes difficult for the local solution as can be seen from Example \ref{example local standard but not local real}. There, the point $x=0$ is not a local $l$-minimal solution of the corresponding set-valued formulation.

The following proposition illustrates one situation to go from a local standard optimistic solution to a local $l$-minimal solution of the set-valued formulation.

\begin{proposition}\label{local standard to local l as well local real}
Consider the usual ordering on $\mathbb{R}$ induced by the ordering cone $C=\mathbb{R}_+$. Let $(x^*,y^*)$ be a local standard optimistic solution to the problem (\ref{bilevel real upper}). Then $x^*$ is a local $l$-minimal solution of (\ref{bilevel set-valued upper}) only if $x^*$ is a local real optimistic solution to the problem (\ref{bilevel real upper}).
\end{proposition}
\begin{proof}
The proof will follow from Proposition \ref{l u and real optimistic pessimistic} that we derive next.
\end{proof}

The above proposition illustrates that the conditions needed to go from a local standard optimistic solution to a local $l$-minimal solution are the same as to go to a local real optimistic solution. 

We now see how real optimistic and real pessimistic solutions are connected to the solution of the set-valued formulation.

\begin{proposition}\label{l u and real optimistic pessimistic}
Consider the usual ordering on $\mathbb{R}$ induced by the ordering cone $C=\mathbb{R}_+$. Then every global (local) $l$-minimal solution of (\ref{bilevel set-valued upper}) is a global (local) real optimistic solution to the problem (\ref{bilevel real upper}). Similarly, every global (local) $u$-minimal solution of (\ref{bilevel set-valued upper}) is a global (local) real pessimistic solution to the problem (\ref{bilevel real upper}). Converses hold when $F(x,\psi(x))$ are closed valued for all $x\in X$. 
\end{proposition}

\begin{proof}
The proof follows from the observation that when $A,B$ are subsets of $\mathbb{R}$ and $C=\mathbb{R}_+$, then $A\leq_C^l B$ (or $A\leq_C^u B$) implies $\inf A \leq \inf B$ (respectively, $\sup A \leq \sup B$). We do the proof for the local part for $l$-minimal solution and leave the other parts for the readers to derive. Let $x^0$ be a local $l$-minimal solution of (\ref{bilevel set-valued upper}). Then there exists an $\epsilon >0$ such that whenever $x\in X,\; \lvert \lvert x-x^0 \rvert  \rvert <\epsilon$ and $F({x},\psi({x}))\leq_C^l F(x^0,\psi(x^0))$, we get $ F(x^0,\psi(x^0)) \leq_C^l F({x},\psi({x}))$. If possible, assume that $x^0$ is not a local real optimistic solution to the problem (\ref{bilevel real upper}). Then there must exist some $\hat{x}\in X$ with $\lvert \lvert\hat{x}-x^0\lvert \lvert<\epsilon$ such that $F_o(\hat{x})<F_o(x^0)$ (otherwise it would have been a local real optimistic solution for the same $\epsilon$ neighbourhood). Let $y^0\in \psi(x^0)$. Then $F(x^0,y^0)\geq F_o(x^0)>F_o(\hat{x})=\underset{y\in \psi(x)}{\inf}F(x,y)$ and hence, by the definition of infimum, there must exist $\hat{y}\in \psi(\hat{x})$ with $F(\hat{x},\hat{y})<F(x^0,y^0)$. This is true for any $y^0\in \psi(x^0)$ and hence $ F(\hat{x},\psi(\hat{x}))\leq_C^l F(x^0,\psi(x^0))$. But then, by the assumption that $x^0$ is a local $l$-minimal solution, we have  $F(x^0,\psi(x^0)) \leq_C^l F(\hat{x},\psi(\hat{x}))$. Hence $F_o(x^0)\leq F_o(\hat{x})$ and we arrive at a contradiction. Therefore, $x^0$ is a local real optimistic solution to the problem (\ref{bilevel real upper}). The proof of the fact that a local $u$-minimal solution of the set-valued formulation is a local pessimistic solution of the bilevel problem and the proofs of the global parts follow in a similar manner. 

We now focus on the converse. We do it for the global $u$-solution and leave other cases for the reader as they are quite similar. Assume that $F(x,\psi(x))$ are closed valued for all $x\in X$ and let $x^o\in X$ be a global pessimistic solution of the problem (\ref{bilevel real upper}). If possible assume that $x^0$ is not a global $u$-minimal solution of (\ref{bilevel set-valued upper}). Then there exists $\hat{x}\in X$ such that $ F(\hat{x},\psi(\hat{x})) \leq_C^u  F(x^0,\psi(x^0))$ but $F(x^0,\psi(x^0)) \nleq_C^u F(\hat{x},\psi(\hat{x}))$. Now $ F(\hat{x},\psi(\hat{x})) \leq_C^u  F(x^0,\psi(x^0))$ implies $F_p(\hat{x})\leq F_p(x^0)$. But then, as $x^o$ is assumed to be a global pessimistic solution of the problem (\ref{bilevel real upper}), we must have $F_p(\hat{x}) = F_p(x^0)$. Now, two cases can arise. Either $F_p(\hat{x}) = F_p(x^0)=\infty$, or it is finite. We shall show that in both cases $F(x^0,\psi(x^0)) \leq_C^u F(\hat{x},\psi(\hat{x}))$ and arrive at a contradiction. Firstly assume that $F_p(\hat{x}) = F_p(x^0)=\infty$. Then for any $y^0 \in \psi(x^0)$, $F(x^0,y^0)$ can't be an upper bound of $F(\hat{x},\psi(\hat{x}))$ and hence there exists $\hat{y}\in \psi(\hat{x})$ such that $F(x^0,y^0)<F(\hat{x},\hat{y})$. This is true for any $y^0 \in \psi(x^0)$ and hence, $F(x^0,\psi(x^0)) \leq_C^u F(\hat{x},\psi(\hat{x}))$. Now if $F_p(\hat{x}) = F_p(x^0)$ is finite, since it is assumed that  $F(x,\psi(x))$ are closed valued, there must exists $\hat{y}\in \psi(\hat{x})$ such that $F_p(\hat{x})=F(\hat{x},\hat{y})$. Now for any $y^0 \in \psi(x^0)$, $F(x^0,\psi(x^0))\leq F_p(x^0)=F_p(\hat{x})=F(\hat{x},\hat{y})$. Hence, $F(x^0,\psi(x^0)) \leq_C^u F(\hat{x},\psi(\hat{x}))$. But this contradicts our starting assumption that $F(x^0,\psi(x^0)) \nleq_C^u F(\hat{x},\psi(\hat{x}))$. Therefore, under the closedness assumption, every global pessimistic solution of the problem (\ref{bilevel real upper}) must be a global $u$-minimal solution of (\ref{bilevel set-valued upper}).

\end{proof}

It may be seen from the following example that without the closedness assumption, the converses may not be true in general.

\begin{example}
Consider the following bilevel programming problem where the lower-level problem is given by 
\begin{gather*}
    \underset{y}{\text{minimize}} \;f(x,y)=\lfloor x+y \rfloor \\ \text{subject to } y\in [0,1] \text{ and } x,y \in \mathbb{R} \nonumber
\end{gather*}

and the upper-level problem is 
\begin{gather*}
    \underset{x}{\text{`minimize'}} \; F(x,y)=\begin{cases}
        x-y & \text{if } 0\leq x <1\\
        -x & \text{if } x =1
    \end{cases} \\ \text{subject to } x\in [0,1], y\in \psi(x). \nonumber
\end{gather*}
Here $\lfloor .\rfloor$ is the standard greatest integer function.  Here, the lower-level solution set map $\psi$ can be calculated as
$$ \psi(x)=\begin{cases}
   [0,1-x) & \text{ when } 0\leq x<1\\
   [0,1) & \text{ when } x=1
\end{cases}.$$
The optimistic upper-level objective function $F_o$ turns out to be 
$$ F_o(x)=\begin{cases}
    2x-1 & \text{ when } 0\leq x<1\\
   -1 & \text{ when } x=1
\end{cases}.$$
Finally, the set-valued map $F(x,\psi(x))$ can be calculated as
$$ F(x,\psi(x))=\begin{cases}
    (2x-1,x] & \text{ when } 0\leq x<1\\
   [-1,0) & \text{ when } x=1
\end{cases}.$$
It can be seen that $0$ and $1$ are global real optimistic solutions, but only $1$ is a global $l$-minimal solution of the corresponding set-valued formulation. Subsequently, $(1,0)$ is a global standard optimistic solution. The point $0$ does not correspond to any global standard optimistic solution, which explains that value attainment is really necessary for equivalence even in the global case as we mentioned earlier.
\end{example}

\begin{proposition}\label{vector to real global}
Consider the usual ordering on $\mathbb{R}$ induced by the ordering cone $C=\mathbb{R}_+$. If $(x^*,z^*)$ is a global (local) vector solution of {\normalfont(\ref{bilevel set-valued upper})} then $x^*$ is a global (local) real optimistic solution to the problem {\normalfont(\ref{bilevel real upper})} and there exists $y^*\in \psi(x^*)$ such that  $(x^*,y^*)$ is a global (local) standard optimistic solution to the problem (\ref{bilevel real upper}).
\end{proposition}
\begin{proof}
Let $(x^*,z^*)$ be a global vector solution of (\ref{bilevel set-valued upper}). Then there exists $y^*\in \psi(x^*)$ such that $F(x^*,y^*)=z^*$ and $z^* = \text{Min }\underset{x\in X}{\bigcup}F(x,\psi(x))$. But this means $F(x,y) \geq F(x^*,y^*)$ for all $x\in X$ and $(x,y)\in Graph(\psi)$. Therefore $(x^*,y^*)$ is a global standard optimistic solution to the problem (\ref{bilevel real upper}), and hence $x^*$ is a global real optimistic solution to the problem (\ref{bilevel real upper}).

Now, let $(x^*,z^*)$ be a local vector solution of (\ref{bilevel set-valued upper}). Then there exist $y^*\in \psi(x^*)$ and $\epsilon>0$ such that $F(x^*,y^*)=z^*$ and $z^* = \text{Min }\underset{x\in V}{\bigcup}F(x,\psi(x))$ where $V=X\cap B(x^*,\epsilon)$. This means $ F(x^*,y^*)\leq F(x,y)$ for all $x\in X$ with $\|x-x^*\|<\epsilon$ and $(x,y)\in Graph(\psi)$. But then $F(x,y) \geq F(x^*,y^*)$ for all  $(x,y)\in Graph(\psi)$ such that $\|(x,y)-(x^*,y^*)\|<\epsilon$. Therefore, $(x^*,y^*)$ is a local standard optimistic solution to the problem (\ref{bilevel real upper}). Also from  $F(x^*,y^*)\leq F(x,y)$ for all $x\in X$ with $\|x-x^*\|<\epsilon$ and $(x,y)\in Graph(\psi)$, we get, $F_o(x^*)\leq F_o(x)$ for all $x\in X$ with $\|x-x^*\|<\epsilon$. Therefore, $x^*$ is a local real optimistic solution to the problem {\normalfont(\ref{bilevel real upper})}. 

\end{proof}

 The last few propositions illustrate how solutions to the set-valued formulation of the bilevel problem are closely connected to the existing solution concepts of bilevel programming. 
 The vector-type solutions to the set-valued formulation are intrinsically connected with the standard optimistic solutions whereas the set-type solutions are greatly connected to the real optimistic and real pessimistic solutions, respectively. In fact, there is a one-to-one correspondence between $x$-coordinates of global standard optimistic solutions and $x$-coordinates of global vector solutions of the set-valued formulation. Although every real optimistic (real pessimistic) solution may not correspond to an $l$-minimal (respectively, a $u$-minimal) solution of the set-valued formulation, the next few results say that their existence coincides.

\begin{theorem}\label{global l-minimal and real optimistic coincidence of existence}
Consider the usual ordering on $\mathbb{R}$ induced by the ordering cone $C=\mathbb{R}_+$. Consider the bilevel programming problem (\ref{bilevel real upper}). Let $\mathcal{Q}$ denote the set of all global real optimistic solutions of the problem (\ref{bilevel real upper}) and assume that $\mathcal{Q}$ is nonempty. Define the set $$\mathcal{T}=\lbrace x\in \mathcal{Q} \mid \text{ there exists y in $\psi(x)$ with }F_o(x)=F(x,y)\rbrace.$$ 
\begin{itemize}
    \item[(I)] If $\mathcal{T}$ is empty, then every point in $\mathcal{Q}$ is a global $l$-minimal solution of (\ref{bilevel set-valued upper}).
    \item[(II)] If $\mathcal{T}$ is nonempty, then every point in $\mathcal{T}$ is a global $l$-minimal solution of (\ref{bilevel set-valued upper}) and no element of $\mathcal{Q}\setminus \mathcal{T}$ is a global $l$-minimal solution of (\ref{bilevel set-valued upper}).
\end{itemize}
\end{theorem}

\begin{proof}~\begin{itemize}
    \item[(I)] Assume that $\mathcal{T}$ is empty. Let $x^0\in \mathcal{Q}$. Assume that for some $\hat{x}\in X$ we have $F(\hat{x},\psi(\hat{x}))\leq_C^l F(x^0,\psi(x^0))$. Then $F_o(\hat{x})\leq F_o(x^0)$ and since $x^0\in \mathcal{Q}$, we must have $F_o(x^0)\leq F_o(\hat{x})$, and hence $F_o(\hat{x})= F_o(x^0)$, implying that $\hat{x}\in \mathcal{Q}$. Let $\hat{y}\in \psi(\hat{x})$. Since $\mathcal{T}=\emptyset$, $\hat{x} \not\in \mathcal{T}$ and hence, $$F(\hat{x},\hat{y})>F_o(\hat{x})=F_o(x^0)=\underset{y\in \psi(x^0)}{\inf}F(x^0,y).$$ Then, by the definition of infimum, there must exist $y^0\in \psi(x^0)$ with $$F_o(x^0)\leq F(x^0,y^0)< F(\hat{x},\hat{y}). $$ Thus for any $\hat{y}\in \psi(\hat{x})$ there exists $y^0\in \psi(x^0)$ such that $ F(x^0,y^0)< F(\hat{x},\hat{y}) $. Hence, $F(x^0,\psi(x^0))\leq_C^l F(\hat{x},\psi(\hat{x}))$. This is true for any $\hat{x}\in X$ with $F(\hat{x},\psi(\hat{x}))\leq_C^l F(x^0,\psi(x^0))$. Thus, $x^0$ is a global $l$-minimal solution of (\ref{bilevel set-valued upper}).  

 \item[(II)] Now, assume that $\mathcal{T}$ be nonempty. We first show that every point in $\mathcal{T}$ is a global $l$-minimal solution of (\ref{bilevel set-valued upper}).  In fact, for any $x^0\in \mathcal{T}$, there exists $y^0\in \psi(x^0)$ such that, $F_o(x^0)=F(x^0,y^0)$. Hence, by Proposition \ref{real to standard local global}, $(x^0,y^0)$ is a global standard optimistic solution, and hence from Proposition \ref{global standard to global l}, we conclude that $x^0$ is a global $l$-minimal solution of (\ref{bilevel set-valued upper}).

 To complete the proof, we now show that no element of $\mathcal{Q}\setminus \mathcal{T}$ is a global $l$-minimal solution of (\ref{bilevel set-valued upper}). When $\mathcal{Q}\setminus \mathcal{T}=\emptyset$, it is vacuously satisfied. Therefore consider that $\mathcal{Q}\setminus \mathcal{T}\neq \emptyset$. Let $x^0\in \mathcal{Q}\setminus \mathcal{T}$. Then for any $\hat{x}\in \mathcal{T}$, since both $x^0,\hat{x}\in \mathcal{Q}$, we have $F_o(\hat{x})=F_o(x^0).$ Since, $\hat{x}\in \mathcal{T}$, there exists $\hat{y}\in \psi(\hat{x})$ such that $F_o(\hat{x})=F(\hat{x},\hat{y})$. Then, for any $y^0\in \psi(x^0)$, since $x^0\in \mathcal{Q}\setminus \mathcal{T}$, 
 $$F(x^0,y^0)>F_o(x^0)=F_o(\hat{x})=F(\hat{x},\hat{y}).$$
 But this implies $F(\hat{x},\psi(\hat{x}))\leq_C^l F(x^0,\psi(x^0))$, but $F(x^0,\psi(x^0))\not \leq_C^l F(\hat{x},\psi(\hat{x}))$. Thus  $x^0$ is not a global $l$-minimal solution of (\ref{bilevel set-valued upper})
 \end{itemize}
 \end{proof}

 We now derive a similar result between local real optimistic solutions of the bilevel problem and local $l$-minimal solutions of the set-valued formulation.

 \begin{theorem}\label{local l-minimal and real optimistic coincidence of existence}
Consider the usual ordering on $\mathbb{R}$ induced by the ordering cone $C=\mathbb{R}_+$. Consider the bilevel programming problem (\ref{bilevel real upper}). Let the problem (\ref{bilevel real upper}) have a local real optimistic solution.
\begin{itemize}
    \item[(I)] If there exists at least one local real optimistic solution $x^0$ of the problem (\ref{bilevel real upper}) such that there is one $y^0\in \psi(x^0)$ with $F_o(x^0)=F(x^0,y^0)$, then $x^0$ is a local $l$-minimal solution of (\ref{bilevel set-valued upper}).
    \item[(II)] If no local real optimistic solution $x$ has an $y\in \psi(x)$ with $F_o(x)=F(x,y)$, then every local real optimistic solution of problem (\ref{bilevel real upper}) is also a local $l$-minimal solution of (\ref{bilevel set-valued upper}).
\end{itemize}
\end{theorem}

\begin{proof}
 ~   

\begin{itemize}
    \item[(I)] Consider the case that there exists one local real optimistic solution $x^0\in X$ such that there is one $y^0\in \psi(x^0)$ with $F_o(x^0)=F(x^0,y^0)$. Since $x^0$ is a local real optimistic solution of the problem (\ref{bilevel real upper}), there must exists some $\epsilon >0$ such that $F_o(x^0)\leq F_o(x)$ for all $x\in X$ with $\|x-x^0\|<\epsilon$. Then for any $x\in X$ with $\|x-x^0\|<\epsilon$ and for any $y\in \psi(x)$, $$F(x,y)\geq F_o(x) \geq F_o(x^0)=F(x^0,y^0). $$ This implies $F(x^0,\psi(x^0))\leq_C^l F({x},\psi({x}))$ for any $x\in X$ with $\|x-x^0\|<\epsilon$. Hence, $x^0$ is a local $l$-minimal solution of (\ref{bilevel set-valued upper}). 

\item[(II)] Now suppose that the other case happens, that is, no local real optimistic solution $x$ has some $y\in \psi(x)$ with $F_o(x)=F(x,y)$. Let $x^0\in X$ be a local real optimistic solution of the problem (\ref{bilevel real upper}). Then there must exists some $\epsilon >0$ such that $F_o(x^0)\leq F_o(x)$ for all $x\in X$ with $\|x-x^0\|<\epsilon$. Suppose, if possible, that $x^0$ is not a local $l$-minimal solution of (\ref{bilevel set-valued upper}). Then, for the above $\epsilon$, there exists $\hat{x}\in X$ such that $\|\hat{x}-x^0\|< \epsilon$ and $F(\hat{x},\psi(\hat{x}))\leq_C^l F(x^0,\psi(x^0))$ but $F(x^0,\psi(x^0))\nleq_C^l F(\hat{x},\psi(\hat{x}))$. Now $F(\hat{x},\psi(\hat{x}))\leq_C^l F(x^0,\psi(x^0))$ implies $F_o(\hat{x})\leq F_o(x^0)$. Since $x^0$ is assumed to be a local real optimistic solution of the problem (\ref{bilevel real upper}) with the above $\epsilon$, we must have $F_o(\hat{x})= F_o(x^0)$. Hence, $\hat{x}$ is also a local real optimistic solution of the problem (\ref{bilevel real upper}). Let $\hat{y}\in \psi(\hat{x})$. Since no local real optimistic solution has the value attainment property by our assumption, $$F(\hat{x},\hat{y})>F_o(\hat{x})=F_o(x^0)=\underset{y\in \psi(x^0)}{\inf}F(x^0,y).$$ Then, by the definition of infimum, there must exist $y^0\in \psi(x^0)$ with $F_o(x^0)\leq F(x^0,y^0) < F(\hat{x},\hat{y})$. Thus for any $\hat{y}\in \psi(\hat{x})$ there exists $y^0\in \psi(x^0)$ such that $ F(x^0,y^0)< F(\hat{x},\hat{y}) $. Hence, $F(x^0,\psi(x^0))\leq_C^l F(\hat{x},\psi(\hat{x}))$. But that is a contradiction. Hence $x^0$ must be a local $l$-minimal solution of (\ref{bilevel set-valued upper}).

\end{itemize}
\end{proof}

We now present an interesting example.
\begin{example}
Consider the following bilevel programming problem where the lower-level problem is given by 
\begin{gather*}
    \underset{y}{\text{minimize}} \;f(x,y)=xy \\ \text{subject to } y\in [0,1] \text{ and } x,y \in \mathbb{R} \nonumber
\end{gather*}

and the upper-level problem is 
\begin{gather*}
    \underset{x}{\text{`minimize'}} \; F(x,y)=\begin{cases}
        x-y & \text{if } y>\frac{1}{2}\\
        x^2+y^2 & \text{if } y\leq \frac{1}{2}
    \end{cases} \\ \text{subject to } x\in [-1,1], y\in \psi(x) . \nonumber
\end{gather*}
Here, the lower-level solution set map $\psi$ can be calculated as
$$ \psi(x)=\begin{cases}
   \{1\} & \text{ when } -1\leq x<0\\
   [0,1] & \text{ when } x=0\\
   \{0\} & \text{ when } 0< x\leq 1\\
\end{cases}.$$
The optimistic upper-level objective function $F_o$ turns out to be 
$$ F_o(x)=\begin{cases}
   x-1 & \text{ when } -1\leq x<0\\
   -1 & \text{ when } x=0\\
   x^2 & \text{ when } 0< x\leq 1\\
\end{cases}.$$
Finally, the set-valued map $F(x,\psi(x))$ can be calculated as
$$ F(x,\psi(x))=\begin{cases}
   \{x-1\} & \text{ when } -1\leq x<0\\
   [0,\frac{1}{4}]\cup[-1,-\frac{1}{2}) & \text{ when } x=0\\
   \{x^2\} & \text{ when } 0< x\leq 1\\
\end{cases}.$$
It can be seen that $x=-1$ is a global (and hence local) real optimistic solution of the bilevel problem, as well as a global (and hence local) $l$-minimal solution of the set-valued formulation. Correspondingly, $(-1,1)$ is a global (and hence local) standard optimistic solution. Also $(0,0)$ is a local standard optimistic solution of the bilevel problem (can be shown by considering a circular neighbourhood of radius say $\frac{1}{3}$ around the point $(0,0)$). Interestingly, $F_o(0)=-1$ which is achieved at $(0,-1)$. But $x=0$ is neither a local real optimistic solution of the bilevel problem nor a local $l$-minimal solution of the set-valued formulation. 
\end{example}
The above example shows that when considering local solutions, only real optimistic solutions (and not the standard optimistic ones) are closely connected with $l$-minimal solutions of the set-valued formulation.

We now derive a similar connection between real pessimistic solutions of the bilevel problem and $u$-minimal solutions of the set-valued formulation. 

\begin{theorem}\label{global u-minimal and real pessimistic coincidence of existence}
Consider the usual ordering on $\mathbb{R}$ induced by the ordering cone $C=\mathbb{R}_+$. Consider the bilevel programming problem (\ref{bilevel real upper}). Let $\hat{\mathcal{Q}}$ denote the set of all global real pessimistic solutions of the problem (\ref{bilevel real upper}) and assume that $\hat{\mathcal{Q}}$ is nonempty. Define the set $$\hat{\mathcal{T}}=\lbrace x\in \hat{\mathcal{Q}} \mid \text{ there exists y in $\psi(x)$ with }F_p(x)=F(x,y)\rbrace.$$ 
\begin{itemize}
    \item[(I)] If $\hat{\mathcal{Q}}\setminus \hat{\mathcal{T}}=\emptyset$, then every point in $\hat{\mathcal{Q}}$ is a global $u$-minimal solution of (\ref{bilevel set-valued upper}).
    \item[(II)] If $\hat{\mathcal{Q}}\setminus \hat{\mathcal{T}}\neq \emptyset$, then every point in $\hat{\mathcal{Q}}\setminus \hat{\mathcal{T}}$ is a global $u$-minimal solution of (\ref{bilevel set-valued upper}) and no element of $\hat{\mathcal{T}}$ is a global $u$-minimal solution of (\ref{bilevel set-valued upper}).
\end{itemize}
\end{theorem}

\begin{proof}
 ~   

\begin{itemize}
    \item[(I)] Consider the case that $\hat{\mathcal{Q}}\setminus \hat{\mathcal{T}}=\emptyset$. We claim that in this case every point in $\hat{\mathcal{Q}}$ is a global $u$-minimal solution of (\ref{bilevel set-valued upper}). To justify our claim, let $x^0\in \hat{\mathcal{Q}}$. Then $x^0\in \hat{\mathcal{T}}$ and hence there exists $y^0\in \psi(x^0)$ such that $F_p(x^0)=F(x^0,y^0)$. Let $\hat{x}\in X$ be such that $F(\hat{x},\psi(\hat{x}))\leq_C^u F(x^0,\psi(x^0))$. Then $F_p(\hat{x})\leq F_p(x^0)$ and since $x^0\in \hat{\mathcal{Q}}$, we must have $F_p(x^0)\leq F_p(\hat{x})$ and hence $F_p(\hat{x})=F_p(x^0)$ implying that $\hat{x}\in \hat{\mathcal{Q}}=\hat{\mathcal{T}}$. Then, by definition of $\hat{T}$, there exists $\hat{y}\in \psi(\hat{x})$ such that $F_p(\hat{x})=F(\hat{x},\hat{y})$.
Let $y\in \psi(x^0)$. Then 
$$ F(x^0,y)\leq F_p(x^0)=F_p(\hat{x})=F(\hat{x},\hat{y}). $$
This is true for any $y\in \psi(x^0)$ and hence $F(x^0,\psi(x^0)) \leq_C^u F(\hat{x},\psi(\hat{x}))$. Thus, for any $\hat{x}\in X$ whenever $F(\hat{x},\psi(\hat{x}))\leq_C^u F(x^0,\psi(x^0))$, we have $F(x^0,\psi(x^0)) \leq_C^u F(\hat{x},\psi(\hat{x}))$. Therefore, $x^0$ is a global $u$-minimal solution of (\ref{bilevel set-valued upper}). 

\item[(II)]Now, consider the case that $\hat{\mathcal{Q}}\setminus \hat{\mathcal{T}}\neq \emptyset$. In this case, we claim that every point in $\hat{\mathcal{Q}}\setminus \hat{\mathcal{T}}$ is a global $u$-minimal solution of (\ref{bilevel set-valued upper}).  Let $x^0\in \hat{\mathcal{Q}}\setminus \hat{\mathcal{T}}$. Consider any $\hat{x}\in X$ such that $F(\hat{x},\psi(\hat{x}))\leq_C^u F(x^0,\psi(x^0))$. Then $F_p(\hat{x})\leq F_p(x^0)$ and since $x^0\in \hat{\mathcal{Q}}$, we must have $F_p(x^0)\leq F_p(\hat{x})$ implying that $F_p(\hat{x})=F_p(x^0)$ and hence $\hat{x}\in \hat{\mathcal{Q}}$. Let $y^0\in \psi(x^0)$. Since $x^0 \notin \hat{T}$, 
$$F(x^0,y^0)<F_p(x^0)=F_p(\hat{x})=\underset{y\in \psi(\hat{x})}{\sup}F(\hat{x},y).$$
Then, by the definition of supremum, there exists $\hat{y}\in \psi(\hat{x})$ such that $F(x^0,y^0)<F(\hat{x},\hat{y})=F_p(\hat{x})$. Thus for any $y^0\in \psi(x^0)$, we can find $\hat{y}\in \psi(\hat{x})$ such that  $F(x^0,y^0)<F(\hat{x},\hat{y})$, and hence $F(x^0,\psi(x^0)) \leq_C^u F(\hat{x},\psi(\hat{x}))$. Thus, for any $\hat{x}\in X$ whenever $F(\hat{x},\psi(\hat{x}))\leq_C^u F(x^0,\psi(x^0))$, we have $F(x^0,\psi(x^0)) \leq_C^u F(\hat{x},\psi(\hat{x}))$. Therefore, $x^0$ is a global $u$-minimal solution of (\ref{bilevel set-valued upper}).

To complete the proof, we have to now show that in this case no element of $\hat{\mathcal{T}}$ is a global $u$-minimal solution of (\ref{bilevel set-valued upper}). If $\hat{\mathcal{T}}=\emptyset$, it is vacuously satisfied. So assume that $\hat{\mathcal{T}}\neq \emptyset$. Let $x^0\in \hat{\mathcal{T}}$. Then for any $\hat{x}\in \hat{\mathcal{Q}}\setminus \hat{\mathcal{T}}$, since both $\hat{x}$ and $x^0$ are in $\hat{\mathcal{Q}}$, $F_p(\hat{x})=F_p(x^0)$. Moreover, since $x^0\in \hat{\mathcal{T}}$, there exists $y^0\in \psi(x^0)$ such that $F_p(x^0)=F(x^0,y^0)$. Now, for any $\hat{y} \in \psi(\hat{x})$, since $\hat{x}\in \hat{\mathcal{Q}}\setminus \hat{\mathcal{T}}$, $$F(\hat{x},\hat{y})<F_p(\hat{x})=F_p(x^0)=F(x^0,y^0).$$
But this implies that $F(\hat{x},\psi(\hat{x}))\leq_C^l F(x^0,\psi(x^0))$ but $F(x^0,\psi(x^0)) \not \leq_C^u F(\hat{x},\psi(\hat{x}))$. Thus, $x^0$ is not a global $u$-minimal solution of (\ref{bilevel set-valued upper})

\end{itemize}
\end{proof}

A similar result holds true between local real pessimistic solutions and local $u$-minimal solutions of the set-valued formulation. We just write the result and leave the proof for the reader to derive.

\begin{theorem}\label{local u-minimal and real pessimistic coincidence of existence}
Consider the usual ordering on $\mathbb{R}$ induced by the ordering cone $C=\mathbb{R}_+$. Consider the bilevel programming problem (\ref{bilevel real upper}). Let the problem (\ref{bilevel real upper}) have a local real pessimistic solution.
\begin{itemize}
    \item[(I)] If for every local real pessimistic solution $x$ of the problem (\ref{bilevel real upper}) there is one $y\in \psi(x)$ with $F_p(x)=F(x,y)$, then every local real pessimistic solution of the problem (\ref{bilevel real upper}) is a local $u$-minimal solution of (\ref{bilevel set-valued upper}).
    \item[(II)] If for some local real pessimistic solution $x^0$ there is no $y\in \psi(x^0)$ with $F_p(x)=F(x^0,y)$, then $x^0$ is a local $u$-minimal solution of (\ref{bilevel set-valued upper}).
\end{itemize}
\end{theorem}

\section{Conclusions}\label{conclusion}
Let us now summarize what we believe to be the key ideas discussed in this article. Our main goal was to show that the formulation of the bilevel problem as a set-valued optimization problem is deeply connected with the optimistic and pessimistic formulation of the problem. The set-valued formulation is the most natural single-level reformulation of the bilevel programming problem if the lower-level problem does not have a unique solution for each upper-level decision variable. The ideas of the optimistic and the pessimistic formulations came in solely to avoid the set-valued optimization formulation of the bilevel programming problem. Though both the optimistic and pessimistic approaches appear very logical and possibly a natural thing to do, it was not clear if at all they have some relationship with the set-valued formulation. Such a relationship would further strengthen the use of the optimistic and pessimistic formulation.

In this article, we show that the $l$-minimal and $u$-minimal solutions of the set-valued formulation are deeply linked with the real optimistic and real pessimistic formulation, respectively. 
In effect, we show that computing a real optimistic solution or a real pessimistic solution corresponds to the computing of an $l$-minimal or a $u$-minimal solution of the set-valued optimization formulation. This also possibly points out the fact that these two formulations are the only possible ways to avoid the set-valued situation while computing the solution of a bilevel optimization. Our study thus uses the set-valued optimization formulation as an explanatory tool to get deeper insights into the nature of the optimistic and pessimistic formulations of the bilevel problems.

\backmatter

\bmhead{Acknowledgments} The first author thanks SERB (File No PDF/2022/001905) for the financial support.

\section*{Declarations}
\begin{itemize}
\item Conflict of interest/Competing interests : There is no conflict of interest among the authors.
\item Availability of data and materials : NA
\item Code availability : NA
\item Authors' contributions : All the authors have contributed equally.
\end{itemize}


\bibliography{sn-bibliography}


\begin{thebibliography}{35}
\ifx \bisbn   \undefined \def \bisbn  #1{ISBN #1}\fi
\ifx \binits  \undefined \def \binits#1{#1}\fi
\ifx \bauthor  \undefined \def \bauthor#1{#1}\fi
\ifx \batitle  \undefined \def \batitle#1{#1}\fi
\ifx \bjtitle  \undefined \def \bjtitle#1{#1}\fi
\ifx \bvolume  \undefined \def \bvolume#1{\textbf{#1}}\fi
\ifx \byear  \undefined \def \byear#1{#1}\fi
\ifx \bissue  \undefined \def \bissue#1{#1}\fi
\ifx \bfpage  \undefined \def \bfpage#1{#1}\fi
\ifx \blpage  \undefined \def \blpage #1{#1}\fi
\ifx \burl  \undefined \def \burl#1{\textsf{#1}}\fi
\ifx \doiurl  \undefined \def \doiurl#1{\url{https://doi.org/#1}}\fi
\ifx \betal  \undefined \def \betal{\textit{et al.}}\fi
\ifx \binstitute  \undefined \def \binstitute#1{#1}\fi
\ifx \binstitutionaled  \undefined \def \binstitutionaled#1{#1}\fi
\ifx \bctitle  \undefined \def \bctitle#1{#1}\fi
\ifx \beditor  \undefined \def \beditor#1{#1}\fi
\ifx \bpublisher  \undefined \def \bpublisher#1{#1}\fi
\ifx \bbtitle  \undefined \def \bbtitle#1{#1}\fi
\ifx \bedition  \undefined \def \bedition#1{#1}\fi
\ifx \bseriesno  \undefined \def \bseriesno#1{#1}\fi
\ifx \blocation  \undefined \def \blocation#1{#1}\fi
\ifx \bsertitle  \undefined \def \bsertitle#1{#1}\fi
\ifx \bsnm \undefined \def \bsnm#1{#1}\fi
\ifx \bsuffix \undefined \def \bsuffix#1{#1}\fi
\ifx \bparticle \undefined \def \bparticle#1{#1}\fi
\ifx \barticle \undefined \def \barticle#1{#1}\fi
\bibcommenthead
\ifx \bconfdate \undefined \def \bconfdate #1{#1}\fi
\ifx \botherref \undefined \def \botherref #1{#1}\fi
\ifx \url \undefined \def \url#1{\textsf{#1}}\fi
\ifx \bchapter \undefined \def \bchapter#1{#1}\fi
\ifx \bbook \undefined \def \bbook#1{#1}\fi
\ifx \bcomment \undefined \def \bcomment#1{#1}\fi
\ifx \oauthor \undefined \def \oauthor#1{#1}\fi
\ifx \citeauthoryear \undefined \def \citeauthoryear#1{#1}\fi
\ifx \endbibitem  \undefined \def \endbibitem {}\fi
\ifx \bconflocation  \undefined \def \bconflocation#1{#1}\fi
\ifx \arxivurl  \undefined \def \arxivurl#1{\textsf{#1}}\fi
\csname PreBibitemsHook\endcsname

\bibitem{Dempe2002foundation}
\begin{bbook}
\bauthor{\bsnm{Dempe}, \binits{S.}}:
\bbtitle{Foundations of Bilevel Programming}.
\bsertitle{Nonconvex Optimization and its Applications},
vol. \bseriesno{61},
p. \bfpage{306}.
\bpublisher{Kluwer Academic Publishers},
\blocation{Dordrecht}
(\byear{2002}).
\doiurl{10.1007/b101970}
\end{bbook}
\endbibitem

\bibitem{Dutta2017optimalityofbilevelthroughvariational}
\begin{bchapter}
\bauthor{\bsnm{Dutta}, \binits{J.}}:
\bctitle{Optimality conditions for bilevel programming: an approach through variational analysis}.
In: \bbtitle{Generalized {N}ash Equilibrium Problems, Bilevel Programming and {MPEC}}.
\bsertitle{Forum Interdiscip. Math.},
pp. \bfpage{43}--\blpage{64}.
\bpublisher{Springer},
\blocation{Singapore}
(\byear{2017})
\end{bchapter}
\endbibitem

\bibitem{Zemkohoillposed}
\begin{barticle}
\bauthor{\bsnm{Zemkoho}, \binits{A.B.}}:
\batitle{Solving ill-posed bilevel programs}.
\bjtitle{Set-Valued Var. Anal.}
\bvolume{24}(\bissue{3}),
\bfpage{423}--\blpage{448}
(\byear{2016}).
\doiurl{10.1007/s11228-016-0371-x}
\end{barticle}
\endbibitem

\bibitem{pileckathesis}
\begin{botherref}
\oauthor{\bsnm{Pilecka}, \binits{M.}}:
Set-valued optimization and its application to bilevel optimization.
PhD thesis,
TU Bergakademie Freiberg,
Freiberg, Germany
(2015)
\end{botherref}
\endbibitem

\bibitem{stackelberg1952theory}
\begin{bbook}
\bauthor{\bsnm{Stackelberg}, \binits{H.}}, \betal:
\bbtitle{Theory of the Market Economy}.
\bpublisher{Oxford University Press},
\blocation{Oxford}
(\byear{1952})
\end{bbook}
\endbibitem

\bibitem{Bard}
\begin{bbook}
\bauthor{\bsnm{Bard}, \binits{J.F.}}:
\bbtitle{Practical Bilevel optimization--Algorithms and Applications}.
\bsertitle{Nonconvex Optimization and its Applications},
vol. \bseriesno{30},
p. \bfpage{473}.
\bpublisher{Kluwer Academic Publishers},
\blocation{Dordrecht}
(\byear{1998}).
\doiurl{10.1007/978-1-4757-2836-1}
\end{bbook}
\endbibitem

\bibitem{Bard1988}
\begin{barticle}
\bauthor{\bsnm{Bard}, \binits{J.F.}}:
\batitle{Convex two-level optimization}.
\bjtitle{Math. Programming}
\bvolume{40}(\bissue{1, (Ser. A)}),
\bfpage{15}--\blpage{27}
(\byear{1988}).
\doiurl{10.1007/BF01580720}
\end{barticle}
\endbibitem

\bibitem{dempe1992}
\begin{barticle}
\bauthor{\bsnm{Dempe}, \binits{S.}}:
\batitle{A necessary and a sufficient optimality condition for bilevel programming problems}.
\bjtitle{Optimization}
\bvolume{25}(\bissue{4}),
\bfpage{341}--\blpage{354}
(\byear{1992}).
\doiurl{10.1080/02331939208843831}
\end{barticle}
\endbibitem

\bibitem{dempe2006dutta}
\begin{barticle}
\bauthor{\bsnm{Dempe}, \binits{S.}},
\bauthor{\bsnm{Dutta}, \binits{J.}},
\bauthor{\bsnm{Lohse}, \binits{S.}}:
\batitle{Optimality conditions for bilevel programming problems}.
\bjtitle{Optimization}
\bvolume{55}(\bissue{5-6}),
\bfpage{505}--\blpage{524}
(\byear{2006}).
\doiurl{10.1080/02331930600816189}
\end{barticle}
\endbibitem

\bibitem{Dempe2012bilevelandcomplimentary}
\begin{barticle}
\bauthor{\bsnm{Dempe}, \binits{S.}},
\bauthor{\bsnm{Dutta}, \binits{J.}}:
\batitle{Is bilevel programming a special case of a mathematical program with complementarity constraints?}
\bjtitle{Math. Program.}
\bvolume{131}(\bissue{1-2, Ser. A}),
\bfpage{37}--\blpage{48}
(\byear{2012}).
\doiurl{10.1007/s10107-010-0342-1}
\end{barticle}
\endbibitem

\bibitem{Dempedutta2006bilevelwithconvexlower}
\begin{bchapter}
\bauthor{\bsnm{Dutta}, \binits{J.}},
\bauthor{\bsnm{Dempe}, \binits{S.}}:
\bctitle{Bilevel programming with convex lower level problems}.
In: \bbtitle{Optimization with Multivalued Mappings}.
\bsertitle{Springer Optim. Appl.},
vol. \bseriesno{2},
pp. \bfpage{51}--\blpage{71}.
\bpublisher{Springer},
\blocation{New York}
(\byear{2006}).
\doiurl{10.1007/0-387-34221-4\_3}
\end{bchapter}
\endbibitem

\bibitem{borwein1977multivalued}
\begin{barticle}
\bauthor{\bsnm{Borwein}, \binits{J.}}:
\batitle{Multivalued convexity and optimization: a unified approach to inequality and equality constraints}.
\bjtitle{Math. Programming}
\bvolume{13}(\bissue{2}),
\bfpage{183}--\blpage{199}
(\byear{1977}).
\doiurl{10.1007/BF01584336}
\end{barticle}
\endbibitem

\bibitem{corley1988optimality}
\begin{barticle}
\bauthor{\bsnm{Corley}, \binits{H.W.}}:
\batitle{Optimality conditions for maximizations of set-valued functions}.
\bjtitle{J. Optim. Theory Appl.}
\bvolume{58}(\bissue{1}),
\bfpage{1}--\blpage{10}
(\byear{1988}).
\doiurl{10.1007/BF00939767}
\end{barticle}
\endbibitem

\bibitem{DK4}
\begin{barticle}
\bauthor{\bsnm{Kuroiwa}, \binits{D.}}:
\batitle{Existence theorems of set optimization with set-valued maps}.
\bjtitle{J. Inf. Optim. Sci.}
\bvolume{24}(\bissue{1}),
\bfpage{73}--\blpage{84}
(\byear{2003}).
\doiurl{10.1080/02522667.2003.10699556}
\end{barticle}
\endbibitem

\bibitem{Akhtar}
\begin{bbook}
\bauthor{\bsnm{Khan}, \binits{A.A.}},
\bauthor{\bsnm{Tammer}, \binits{C.}},
\bauthor{\bsnm{Z\u{a}linescu}, \binits{C.}}:
\bbtitle{Set-valued optimization---An Introduction with Applications}.
\bsertitle{Vector Optimization},
p. \bfpage{765}.
\bpublisher{Springer},
\blocation{Heidelberg}
(\byear{2015}).
\doiurl{10.1007/978-3-642-54265-7}
\end{bbook}
\endbibitem

\bibitem{Gofert}
\begin{bbook}
\bauthor{\bsnm{G\"{o}pfert}, \binits{A.}},
\bauthor{\bsnm{Riahi}, \binits{H.}},
\bauthor{\bsnm{Tammer}, \binits{C.}},
\bauthor{\bsnm{Z\u{a}linescu}, \binits{C.}}:
\bbtitle{Variational Methods in Partially Ordered Spaces}.
\bsertitle{CMS Books in Mathematics/Ouvrages de Math\'{e}matiques de la SMC},
vol. \bseriesno{17},
p. \bfpage{350}.
\bpublisher{Springer},
\blocation{New York}
(\byear{2003}).
\doiurl{10.1007/b97568}
\end{bbook}
\endbibitem

\bibitem{corley1987existence}
\begin{barticle}
\bauthor{\bsnm{Corley}, \binits{H.W.}}:
\batitle{Existence and {L}agrangian duality for maximizations of set-valued functions}.
\bjtitle{J. Optim. Theory Appl.}
\bvolume{54}(\bissue{3}),
\bfpage{489}--\blpage{501}
(\byear{1987}).
\doiurl{10.1007/BF00940198}
\end{barticle}
\endbibitem

\bibitem{Luc89}
\begin{bbook}
\bauthor{\bsnm{Luc}, \binits{D.T.}}:
\bbtitle{Theory of Vector Optimization}.
\bsertitle{Lecture Notes in Economics and Mathematical Systems},
vol. \bseriesno{319},
p. \bfpage{173}.
\bpublisher{Springer},
\blocation{Berlin}
(\byear{1989}).
\doiurl{10.1007/978-3-642-50280-4}
\end{bbook}
\endbibitem

\bibitem{DK}
\begin{bchapter}
\bauthor{\bsnm{Kuroiwa}, \binits{D.}}:
\bctitle{On set-valued optimization}.
In: \bbtitle{Proceedings of the {T}hird {W}orld {C}ongress of {N}onlinear {A}nalysts, {P}art 2 ({C}atania, 2000)},
vol. \bseriesno{47},
pp. \bfpage{1395}--\blpage{1400}
(\byear{2001}).
\doiurl{10.1016/S0362-546X(01)00274-7}
\end{bchapter}
\endbibitem

\bibitem{KurTanTru97}
\begin{bchapter}
\bauthor{\bsnm{Kuroiwa}, \binits{D.}},
\bauthor{\bsnm{Tanaka}, \binits{T.}},
\bauthor{\bsnm{Ha}, \binits{T.X.D.}}:
\bctitle{On cone convexity of set-valued maps}.
In: \bbtitle{Proceedings of the {S}econd {W}orld {C}ongress of {N}onlinear {A}nalysts, {P}art 3 ({A}thens, 1996)},
vol. \bseriesno{30},
pp. \bfpage{1487}--\blpage{1496}
(\byear{1997}).
\doiurl{10.1016/S0362-546X(97)00213-7}
\end{bchapter}
\endbibitem

\bibitem{H}
\begin{barticle}
\bauthor{\bsnm{Hamel}, \binits{A.H.}},
\bauthor{\bsnm{L\"{o}hne}, \binits{A.}}:
\batitle{A set optimization approach to zero-sum matrix games with multi-dimensional payoffs}.
\bjtitle{Math. Methods Oper. Res.}
\bvolume{88}(\bissue{3}),
\bfpage{369}--\blpage{397}
(\byear{2018}).
\doiurl{10.1007/s00186-018-0639-z}
\end{barticle}
\endbibitem

\bibitem{mirrlees1999theory}
\begin{barticle}
\bauthor{\bsnm{Mirrlees}, \binits{J.A.}}:
\batitle{The theory of moral hazard and unobservable behaviour: Part {I}}.
\bjtitle{The Review of Economic Studies}
\bvolume{66}(\bissue{1}),
\bfpage{3}--\blpage{21}
(\byear{1999})
\end{barticle}
\endbibitem

\bibitem{Leitmann}
\begin{barticle}
\bauthor{\bsnm{Leitmann}, \binits{G.}}:
\batitle{On generalized {S}tackelberg strategies}.
\bjtitle{J. Optim. Theory Appl.}
\bvolume{26}(\bissue{4}),
\bfpage{637}--\blpage{643}
(\byear{1978}).
\doiurl{10.1007/BF00933155}
\end{barticle}
\endbibitem

\bibitem{Osbornegametheorybook1994}
\begin{bbook}
\bauthor{\bsnm{Osborne}, \binits{M.J.}},
\bauthor{\bsnm{Rubinstein}, \binits{A.}}:
\bbtitle{A Course in Game Theory},
p. \bfpage{352}.
\bpublisher{MIT Press},
\blocation{Cambridge, MA}
(\byear{1994})
\end{bbook}
\endbibitem

\bibitem{BNT}
\begin{bbook}
\bauthor{\bsnm{Ben-Tal}, \binits{A.}},
\bauthor{\bsnm{El~Ghaoui}, \binits{L.}},
\bauthor{\bsnm{Nemirovski}, \binits{A.}}:
\bbtitle{Robust Optimization}.
\bsertitle{Princeton Series in Applied Mathematics},
p. \bfpage{542}.
\bpublisher{Princeton University Press},
\blocation{Princeton, NJ}
(\byear{2009}).
\doiurl{10.1515/9781400831050}
\end{bbook}
\endbibitem

\bibitem{Beck}
\begin{barticle}
\bauthor{\bsnm{Beck}, \binits{A.}},
\bauthor{\bsnm{Ben-Tal}, \binits{A.}}:
\batitle{Duality in robust optimization: primal worst equals dual best}.
\bjtitle{Oper. Res. Lett.}
\bvolume{37}(\bissue{1}),
\bfpage{1}--\blpage{6}
(\byear{2009}).
\doiurl{10.1016/j.orl.2008.09.010}
\end{barticle}
\endbibitem

\bibitem{Klamroth17}
\begin{barticle}
\bauthor{\bsnm{Klamroth}, \binits{K.}},
\bauthor{\bsnm{K\"{o}bis}, \binits{E.}},
\bauthor{\bsnm{Sch\"{o}bel}, \binits{A.}},
\bauthor{\bsnm{Tammer}, \binits{C.}}:
\batitle{A unified approach to uncertain optimization}.
\bjtitle{European J. Oper. Res.}
\bvolume{260}(\bissue{2}),
\bfpage{403}--\blpage{420}
(\byear{2017}).
\doiurl{10.1016/j.ejor.2016.12.045}
\end{barticle}
\endbibitem

\bibitem{mehlitz2017contributions}
\begin{botherref}
\oauthor{\bsnm{Mehlitz}, \binits{P.}}:
Contributions to complementarity and bilevel programming in \uppercase{B}anach spaces.
PhD thesis,
TU Bergakademie Freiberg,
Freiberg, Germany
(2017)
\end{botherref}
\endbibitem

\bibitem{borwein1977proper}
\begin{barticle}
\bauthor{\bsnm{Borwein}, \binits{J.}}:
\batitle{Proper efficient points for maximizations with respect to cones}.
\bjtitle{SIAM J. Control Optim.}
\bvolume{15}(\bissue{1}),
\bfpage{57}--\blpage{63}
(\byear{1977}).
\doiurl{10.1137/0315004}
\end{barticle}
\endbibitem

\bibitem{borwein1981lagrange}
\begin{barticle}
\bauthor{\bsnm{Borwein}, \binits{J.M.}}:
\batitle{A {L}agrange multiplier theorem and a sandwich theorem for convex relations}.
\bjtitle{Math. Scand.}
\bvolume{48}(\bissue{2}),
\bfpage{189}--\blpage{204}
(\byear{1981}).
\doiurl{10.7146/math.scand.a-11911}
\end{barticle}
\endbibitem

\bibitem{postolm1986vectorial}
\begin{barticle}
\bauthor{\bsnm{Postolic\u{a}}, \binits{V.}}:
\batitle{Vectorial optimization programs with multifunctions and duality}.
\bjtitle{Ann. Sci. Math. Qu\'{e}bec}
\bvolume{10}(\bissue{1}),
\bfpage{85}--\blpage{102}
(\byear{1986})
\end{barticle}
\endbibitem

\bibitem{DK1}
\begin{barticle}
\bauthor{\bsnm{Kuroiwa}, \binits{D.}}:
\batitle{On natural criteria in set-valued optimization}.
\bjtitle{RIMS Kokyuroku}
\bvolume{1048},
\bfpage{86}--\blpage{92}
(\byear{1998}).
\bcomment{Dynamic decision systems in uncertain environments (Japanese) (Kyoto, 1998)}
\end{barticle}
\endbibitem

\bibitem{berge1997topological}
\begin{bbook}
\bauthor{\bsnm{Berge}, \binits{C.}}:
\bbtitle{Topological spaces---Including a Treatment of Multi-valued Functions, Vector Spaces and Convexity},
p. \bfpage{270}.
\bpublisher{Dover Publications, Inc.},
\blocation{Mineola, NY}
(\byear{1997}).
\bcomment{Translated from the French original by E. M. Patterson, Reprint of the 1963 translation}
\end{bbook}
\endbibitem

\bibitem{dempe2015book}
\begin{bbook}
\bauthor{\bsnm{Dempe}, \binits{S.}},
\bauthor{\bsnm{Kalashnikov}, \binits{V.}},
\bauthor{\bsnm{P\'{e}rez-Vald\'{e}s}, \binits{G.A.}},
\bauthor{\bsnm{Kalashnykova}, \binits{N.}}:
\bbtitle{Bilevel Programming problems--Theory, Algorithms and Applications to Energy Networks}.
\bsertitle{Energy Systems},
p. \bfpage{325}.
\bpublisher{Springer},
\blocation{Heidelberg}
(\byear{2015}).
\doiurl{10.1007/978-3-662-45827-3}
\end{bbook}
\endbibitem

\bibitem{Ausubel1993generalizedmaximumtheorem}
\begin{barticle}
\bauthor{\bsnm{Ausubel}, \binits{L.M.}},
\bauthor{\bsnm{Deneckere}, \binits{R.J.}}:
\batitle{A generalized theorem of the maximum}.
\bjtitle{Econom. Theory}
\bvolume{3}(\bissue{1}),
\bfpage{99}--\blpage{107}
(\byear{1993}).
\doiurl{10.1007/BF01213694}
\end{barticle}
\endbibitem

\end{thebibliography}


\end{document}